\documentclass[12pt,parskip=half]{scrartcl}

\usepackage{latexsym, amssymb,  amsmath, amsthm}

\usepackage[dvipsnames]{xcolor}
\usepackage{graphicx}

\usepackage[shortlabels]{enumitem}
\numberwithin{equation}{section}

\usepackage{mathtools}
\mathtoolsset{showonlyrefs}
\newcommand{\define}{\coloneqq}
\newcommand{\enifed}{\eqqcolon}
\usepackage{bm}

\newtheorem{theorem}{Theorem}[section]
\newtheorem{lemma}[theorem]{Lemma}
\newtheorem{corollary}[theorem]{Corollary}
\newtheorem{proposition}[theorem]{Proposition}

\theoremstyle{definition}

\newtheorem{definition}[theorem]{Definition}
\newtheorem{assumption}[theorem]{Assumption}
\newtheorem{remark}[theorem]{Remark}

\theoremstyle{remark}

\newcommand{\field}[1]{\mathbb{#1}}
\newcommand{\N}{\field{N}}
\newcommand{\R}{\field{R}}
\newcommand{\C}{\field{C}}
\newcommand{\divv}{\operatorname{div}}
\newcommand{\supp}{\mathop\textrm{supp}}

\newcommand{\esssup}{\operatorname{ess\,sup}}
\newcommand{\blangle}{\bigl\langle}
\newcommand{\brangle}{\bigr\rangle}
\newcommand{\eps}{\varepsilon}
\newcommand{\embeds}{\hookrightarrow}
\renewcommand{\Re}{\operatorname{Re}\,}
\renewcommand{\Im}{\operatorname{Im}\,}

\newcommand{\cA}{\mathcal{A}}
\newcommand{\cF}{\mathcal{F}}
\newcommand{\cH}{\mathcal{H}}
\newcommand{\cI}{\mathcal{I}}
\newcommand{\cL}{\mathcal{L}}
\newcommand{\cN}{\mathcal{N}}
\newcommand{\cS}{\mathcal{S}}

\newcommand{\fA}{A}
\newcommand{\fD}{\mathcal{D}}
\newcommand{\fI}{\mathfrak{I}}

\newcommand{\bW}{\mathbb{W}}
\newcommand{\bmW}{\bm{W}}
\newcommand{\bmw}{\bm{w}}
\newcommand{\tA}{\bm{A}}
\newcommand{\tB}{\bm{B}}
\newcommand{\bmu}{\bm{u}}
\newcommand{\bmv}{\bm{v}}
\newcommand{\bmL}{\bm{L}}
\newcommand{\id}{\mathrm{id}}

\usepackage[sorting=nyt,sortcites=true,giveninits=true,backend=biber,isbn=false,url=false]{biblatex}
\usepackage[colorlinks=true,allcolors={Fuchsia}]{hyperref}
\hypersetup{final}

\begin{document}

\begin{center}
  {\large \textbf{Nonautonomous maximal parabolic regularity for nonsmooth quasilinear parabolic systems}}\\[0.5em]
  Hannes Meinlschmidt\footnote{FAU Erlangen-N\"urnberg, Department of
    Mathematics\\Chair for Dynamics, Control, Machine Learning and
    Numerics (AvH Professorship)\\Cauerstra{\ss}e 11, 91058 Erlangen
    (Germany)\\meinlschmidt [at] math.fau.de\\ORCiD 0000-0002-5874-8017}
\end{center}

\begin{abstract}
  In this paper we are concerned with $L^p$-maximal parabolic
  regularity for abstract nonautonomous parabolic systems and their
  quasilinear counterpart in negative Sobolev spaces incorporating
  mixed boundary conditions. Our results are derived in the
  setting of nonsmooth domains with mixed boundary conditions by an
  extrapolation technique which also yields uniform estimates for the
  parabolic solution operators. We require only very mild boundary
  regularity, not in the Lipschitz-class, and generally only bounded
  and measurable complex coefficients. The nonlinear functions in the
  quasilinear formulation can be nonlocal-in-time; this allows also to
  consider certain systems whose stationary counterpart fails to
  satisfy the usual ellipticity conditions.
\end{abstract}


\section[Introduction]{Introduction} \label{sec:intro}

The nonautonomous systems that we consider arise from $m$ parabolic
equations in divergence form for the functions $\bmu= (u_1,\dots,u_m)$,
\begin{equation}\label{eq:mainEqStrong}
  u_i'(t) + L_t^i \bmu(t) = f_i(t), \qquad u_i(0) = u_0^i  \qquad (i=1,\dots,m),
\end{equation}
with 
\begin{equation*}
  L_t^i\bmu(t) \coloneqq - \sum_{j=1}^m \divv\bigl(\fA^{ij}_t \nabla u_j(t) +
  b^{ij}_t u_j(t)\bigr) +\sum_{j=1}^m
  c^{ij}_t\cdot\nabla u_j(t)  + \sum_{j=1}^md^{ij}_tu_j(t),
\end{equation*}
where
the coefficients are complex, bounded and measurable depending on time $t$ and
the spatial variable. (We do not spell out the latter explicitly.) We consider~\eqref{eq:mainEqStrong} on a finite
time interval $(0,T)$ and a bounded domain
$\Omega \subset \R^d$, and each equation shall be cast in a negative
Sobolev space $ W^{-1,q}_{D_i}(\Omega)$ with
$f_i \in L^r(0,T;W^{-1,q}_{D_i}(\Omega))$ for some
$r,q \in (1,\infty)$. The space
$W^{-1,q}_{D_i}(\Omega)$ is the anti-dual space of
$W^{1,q'}_{D_i}(\Omega)$ with $1/q + 1/q' = 1$.  All objects will be
introduced properly in the main text. The index $D_i$ signifies a
homogeneous Dirichlet condition for $u_i$ on the subset
$D_i \subseteq \partial\Omega$ of the boundary, so the abstract
understanding of~\eqref{eq:mainEqStrong} in $W^{-1,q}_{D_i}(\Omega)$
includes mixed boundary conditions; it also allows to consider
inhomogeneous Neumann boundary data for $u_i$ absorbed in $f_i$. The
system will be complemented by initial values. 
Collecting the differential operators $\bmL_t = (L_t^1,\dots,L^m_t)$, we obtain the short hand system
\begin{equation} \label{eq:mainEqNonAuto} 
  \bmu'(t) + \bmL_t \bmu(t) = \bm{f}(t) \qquad \text{in } \bmW^{-1,\bm{q}}_\fD(\Omega),
\end{equation}
with $\bm{f} = (f_1,\dots,f_m) \in L^r(0,T;\bmW^{-1,q}_\fD(\Omega))$ and
$\bmW^{-1,q}_\fD(\Omega) = \prod_{i=1}^m W^{-1,q}_{D_i}(\Omega)$.

Nonautonomous $L^p$-maximal parabolic regularity
for problem~\eqref{eq:mainEqNonAuto} with \emph{constant
  domains} $\bmW^{1,q}_\fD(\Omega)$) means that for every
$f \in L^r(0,T;\bmW^{-1,q}_\fD(\Omega))$, there exists a unique
solution $\bmu$ to~\eqref{eq:mainEqNonAuto} satisfying
\begin{equation}\label{eq:max-reg-reg}
  \bmu \in W^{1,r}\bigl(0,T;\bmW^{-1,q}_\fD(\Omega)\bigr) \cap
  L^r\bigl(0,T;\bmW^{1,q}_\fD(\Omega)\bigr), \qquad \bmu(0) = 0.
\end{equation}
In this paper, we establish this property of $L^p$-maximal parabolic regularity
for nonautonomous systems as in~\eqref{eq:mainEqNonAuto} as a main result. This
is a notoriously difficult question in a nonsmooth setting, which we tackle by
an extrapolation technique from the Hilbert space case $r=q=2$. This way, we
also obtain bounds on the associated parabolic solution operators which are
uniform in the given coefficient data. We pose an ellipticity
assumption in the form of a G{\aa}rding inequality.

In a second step, we then use the maximal regularity results within the
machinery of Amann~\cite{Ama05b} to obtain several wellposedness results in the
maximal regularity class as in~\eqref{eq:max-reg-reg}---both local-in-time and
global-in-time---for abstract quasilinear systems under minimal assumption of
the (short-hand) form
\begin{equation} \label{eq:mainEqQL} \bm{u'}(t) + 
  \bmL(\bmu)_t\bmu(t)
  = \Phi(\bmu)(t) \qquad \text{in}~\bmW^{-1,q}_\fD(\Omega),
  \qquad \bmu(0) = \bm{u_0},
\end{equation}
with
\begin{multline*}
  \bigl[\bmL(\bmu)^i_t\bigr]\bmu(t) \coloneqq - \sum_{j=1}^m \divv\bigl(\fA^{ij}(\bmu)_t \nabla u_j(t) +
  b^{ij}(\bmu)_t u_j(t)\bigr)\\ +\sum_{j=1}^m
  c^{ij}(\bmu)_t\cdot\nabla u_j(t)  + \sum_{j=1}^md^{ij}(\bmu)_tu_j(t).
\end{multline*}
It will be a defining feature of~\eqref{eq:mainEqQL} that both the coefficients
and the right-hand side $\Phi$ are allowed to depend on $\bmu$ in a
\emph{nonlocal-in-time} manner. (This is why we write $\bmL(\bmu)_t$ instead of
$\bmL(\bmu(t))$, and analogously for $\Phi$!) Under certain conditions, this
also allows to treat systems whose stationary counterpart does not satisfy one
of the usual ellipticity conditions; for example, some multi-species models in
chemotaxis have this property.

Let us next first give some context for our results before we coming
to an overview of the content of this work.

\subsubsection*{Context and overview}
The property of nonautonomous $L^p$-maximal parabolic regularity is
recognized as very useful and versatile; first of all, it is of course
an immediate tool to treat linear nonautonomous parabolic evolution
equations and systems, which are ubiquitous in natural sciences---we
refer to the introduction of~\cite{Amann95} and the references
therein---, obtaining sharp and optimal regularity results. It is also
most important when considering quasilinear problems and their
linearizations, which give rise to nonautonomous problems in a natural
way.  Many wellposedness results for quasilinear parabolic evolution
equations build upon this very property by using it for linearizations
of the nonlinear problem in fixed point schemes, see
e.g.~\cite{Ama05b,Pru02}. Sharp results for such linearizations, and,
by transposition, their adjoint versions, are also of great interest
when considering optimal control problems subject to quasilinear
evolution equations as considered in
e.g.~\cite{MMR17a,MMR17b,Hoppe2022,Hoppe2022a,Hoppe2020}. 

In fact, local-in-time wellposedness for---even nonlinear---parabolic systems
with smooth data is well established since long, see for
example~\cite{Ama86,GM87,DHP07,DHP03} or, more recently,
e.g.~\cite{Piasecki2020}. Today, similar results are also
available even in case of highly nonsmooth data for one \emph{single} equation,
cf.~\cite{HR09,ABHR14,tEMR14,MMR17a}. For actual \emph{systems} with nonsmooth
data, there seem to be very little results available---see~\cite{Gallarati2017}
however---, in particular when mixed
boundary conditions are considered.  Such would be of high interest in view of
several interesting and relevant applications in the natural sciences; we refer
to e.g.~\cite{AGH02,Hor10,WinklerTao2015} and the references therein. 
  

This however is a highly nontrivial task because systems lack many properties
which are essential parts of the theory developed for the scalar case. These
include the maximum principle, Gaussian estimates, \emph{a priori} H\"older
estimates, or contractivity properties for the corresponding semigroups in,
say, $L^p(\Omega)$; we refer to~\cite{MK86,KM83,KM99}, to~\cite[Ch.~6
Notes]{Ouhabaz09} and the references therein, to~\cite{Str84,JMS89,JS90}), and
to~\cite{ABBO00} and~\cite[Ch.~4]{CM14}.

The absence of these fundamental tools shows that one must develop new
ideas and instruments to obtain existence, uniqueness and regularity
for parabolic systems, in particular in the context of nonsmooth
data. In this paper, we attempt to do so for a geometric setting where
the spatial domain $\Omega$ may even fail to be Lipschitz, and where
the Dirichlet boundary parts $D_i$ are only required to be
Ahlfors--David regular. (The same setting was assumed in the preceding
works~\cite{HJKR15,ABHR14,Egert2018} and~\cite{DtER17}; see
Assumption~\ref{assu-general} below for details.)
For~\eqref{eq:mainEqNonAuto}, only \emph{boundedness} is supposed for
the coefficients in time and space, whereas we also impose a
continuity at zero time condition for particular cases for the
quasilinear problem~\eqref{eq:mainEqQL}. The starting point is a
classical wellposedness result for parabolic systems with nonsmooth
data: the ``parabolic Lax-Milgram lemma'' of
Lions~\cite[Sect.~XVIII.§3]{DL92} which states nonautonomous
$L^2$-maximal regularity for the linear
problem~\eqref{eq:mainEqNonAuto} in the Hilbert space case, so $r=q=2$
in~\eqref{eq:max-reg-reg}, under minimal assumptions on the data,
assuming a G{\aa}rding inequality for ellipticity. Unfortunately, the
Hilbert space setting is often too restrictive for nonlinear
problems. Thus, based on the main result in~\cite{DtER17}, we use the
\emph{Sneiberg extrapolation theorem} to extrapolate this
nonautonomous maximal regularity property to
$L^r(0,T;\bmW^{-1,q}_\fD(\Omega))$ for pairs $(r,q) \neq (2,2)$. (Note
that $(r,q)$ will in general still be close to $(2,2)$). This is
Theorem~\ref{t-extrapolatsystem} below. The extrapolation technique is
quite abstract but yields a nontrivial result for which we do not know
any other kind of proof in a nonsmooth setting; in fact, we could also
say that the abstract reasoning allows to use the formidable theory
developed for Sobolev spaces incorporating mixed boundary conditions
in a nonsmooth setting in a crucial way. Our technique also admits
explicit bounds on the inverse parabolic operators which are uniform
in the given coefficient data. These are inherited from the Hilbert
space case. 

For $r,q > 2$, the seemingly little gain of regularity compared to the
Hilbert space case $r=q=2$, as in Lions' theorem, is interesting in
itself, but will be particularly useful in the treatment of nonlinear
problems, in particular for space dimension $d=2$. As a simple but
illuminating example, we mention that if $\bmu = (u_1,\dots,u_m)$ is
in, say, $L^\infty(0,T;\bmL^q(\Omega))$---this will in fact be the
case in the maximal regularity context for $r>2$, see
Corollary~\ref{c-vect-Lq}---, then any quadratic function $u_iu_j$ of
components of $\bmu$ will be in
$L^\infty(0,T;\bmW^{-1,q}_\fD(\Omega))$ if and only if $q > d$; so,
one is enabled to include such quadratic nonlinearities in the
$\bmW^{-1,q}_\fD(\Omega)$-framework for $d=2$ exactly when $q>2$.

It is further well known that maximal regularity results for
nonautonomous problems are essential for their quasilinear
counterparts since they give rise to an invertible linearization. We
base our considerations for the quasilinear system~\eqref{eq:mainEqQL}
on a fundamental and very general theorem by Amann~\cite{Ama05b}, see
Theorem~\ref{t-Amann} below. The main assumption there is
nonautonomous maximal parabolic regularity for the system operators
involved, with \emph{constant domains}, which in this context means
that the spatial differential operators $\bmL(\bmu)_t$ must have a
common domain $\bmW^{1,q}_\fD(\Omega)$ in the ambient space
$\bmW^{-1,q}_\fD(\Omega)$ for every time $t \in (0,T)$ and every
function $\bmu$ in the maximal regularity space. Usually, establishing
this property is the fundamental difficulty when applying Amann's
theorem for $q \neq 2$ and there seems to be little hope to emulate the
arguments for the scalar case due to the aforementioned lack of essential tools
in the systems case. However, in the present work, constant domains
are a \emph{built-in feature} of our extrapolation technique. This
allows us to state very general wellposedness results for quasilinear
systems as in~\eqref{eq:mainEqQL} with nonsmooth data,
Theorems~\ref{t-quasilin},~\ref{thm:quasilinear-2}
and~\ref{t-quasilin-global}. The last of these is a global-in-time
wellposedness result building upon the uniform bounds for maximal
parabolic regularity obtained in Theorem~\ref{t-extrapolatsystem}
earlier. In fact, since Amann's theorem (Theorem~\ref{t-Amann}) even
allows for \emph{nonlocal-in-time} dependencies on the state in the
nonlinearities, it is even possible to transform systems which are
formally \emph{non-coercive} into such which are tractable with our
theory. This is explained in Section~\ref{Snonaut6}, and we sketch a 
chemotaxis model application problem in Section~\ref{sec:example}.

\subsubsection*{Outline}

The paper is structured as follows. We first collect several
definitions and prerequisites such as Sneiberg's extrapolation
theorem, fundamental properties of elliptic (system)) operators
in divergence-gradient form, and some results from~\cite{DtER17}
in Section~\ref{sec:prelim}. Basing on these, we establish the
extension of Lions' theorem for the nonautonomous system
problem~\eqref{eq:mainEqNonAuto}, which is the first main
result, in Section~\ref{Snonaut4}. 
The results on nonautonomous maximal
regularity are then used in Section~\ref{Snonaut6} to give
several wellposedness results for the quasilinear
system~\eqref{eq:mainEqQL} basing on Amann's theorem. We
moreover explain how to deal with possibly non-coercive system
operators in this context and on the example of a chemotaxis
application in Section~\ref{sec:example}.

We expect that our notation is standard, otherwise it will be introduced during
the main text. All Banach spaces considered will be complex ones if not
indicated otherwise, and we fix the space dimension $d \geq 2$ and the number
$m \geq 1$ of unknowns and equations in the parabolic systems to be
considered. As a rule of thumb which was already used in the foregoing
introduction, vector-valued objects or operators acting on such will be denoted
in boldface, although we will not be extremely strict with this. This also
includes function spaces such as $\bmL^q(\Omega) \define
L^q(\Omega)^m$. Moreover, for a vector-valued function $\bmu = (u_1,\dots,u_m)$
we consider the gradient $\nabla \bm{u}$ as the aggregation of the individual
gradients into a $\C^{md}$ vector. The standard $\C^{md}$-norm
$|\nabla \bmu(x)|$ is then the same as the Frobenius norm of the
$\C^{m \times d}$-Jacobian $\bm u'(x)$ or its transpose; we mention this
because the latter objects are sometimes used in the literature.

\section{Preliminaries, elliptic operators and systems}\label{sec:prelim}

We collect some preliminary results and definitions for later use. The
first is a highly useful abstract extrapolation result for
continuously invertible operators between complex interpolation
scales,
Sneiberg's famous extrapolation principle. The theorem was first
established in~\cite{S74} and elaborated in more detail
in~\cite{VV88}, but the explicit quantitative estimates as quoted here
were worked out only recently; we refer
to~\cite[Appendix]{Auscher2019}.

\begin{theorem}[Explicit quantitative Sneiberg] \label{t-sneib} Let
  $(X_1, X_2)$ and
  $(Y_1, Y_2)$ be interpolation couples of Banach spaces.  Let further
  $T \in \cL(X_1 \to Y_1) \cap \cL(X_2 \to Y_2)$ and put
  \begin{equation*}
    \gamma\define \|T\|_{X_1 \to Y_1} \vee \|T\|_{X_2 \to
      Y_2}.
  \end{equation*}
  Suppose that for a given $\vartheta \in(0,1)$ the operator
  $T \colon [X_1,X_2]_{\vartheta} \to [Y_1,Y_2]_{\vartheta} $ is a topological
  isomorphism and let
  $\beta = \|T^{-1}\|_{[X_1,X_2]_{\vartheta} \to [Y_1,Y_2]_{\vartheta}}$.
    Then $T \colon [X_1,X_2]_{\theta} \to [Y_1,Y_2]_{\theta}$ is still
    a topological isomorphism for all $\theta \in(0,1)$ satisfying
    \begin{equation*}
      \bigl|\vartheta-\theta\bigr| \leq \frac{1}{6}\,\frac{
        \vartheta \wedge (1-\vartheta)
      }{1 +2 \beta \gamma},
    \end{equation*}
    and the norm of its inverse is bounded by $ 8 \beta$.
\end{theorem}

\subsection{Geometric framework and function
  spaces} \label{sec:ScaEllOp} The geometric setting, which we suppose
to be satisfied in all what follows, is as follows:

\begin{assumption}[Geometry] \label{assu-general} The set $\Omega$ is a
  nonempty bounded and connected open subset of $\R^d$, where $d \geq 2$.  Let
  $D_1,\dots,D_m$ be closed subsets of the boundary $\partial\Omega$
  to be understood as the Dirichlet boundary parts. Set
  $\fD \coloneqq \cap_{i=1}^m D_i$ and $\cN \coloneqq
  \overline{\partial \Omega \setminus \fD}$ and suppose the following:
  \begin{enumerate}
  \item[$(\cN)$] \label{assu-general:i} For every
    $x \in \cN$ there exists an
    open neighbourhood $U_x$ of $x$ in $\R^d$ and a bi-Lipschitz map
    $\zeta_x$ from $ U_x$ onto the cube $K \define  {]{-1},1[}^d$ such
    that $\zeta_x(x) = 0$ and
    \begin{equation*}
      \zeta_x( U_x \cap \Omega) = \bigl\{ x \in K \colon x_d < 0 \bigr\} ,  \qquad
      \zeta_x(U_x \cap \partial \Omega) 
      = \bigl\{ x \in K \colon x_d = 0 \bigr\}.
    \end{equation*}
  \item[$(\fD)$] \label{assu-general:ii} Each $D_i$ is either empty or it satisfies the
    \emph{Ahlfors--David condition}, that is,
    \begin{equation} \label{eq:ahlf} 
      \cH_{d-1}
      \bigl(D_i \cap B(x,r) \bigr) \simeq r^{d-1} 
    \end{equation}
    uniformly for all $x \in D_i$ and $r < 1$, where
    $\mathcal H_{d-1}$ denotes the $(d-1)$-dimensional Hausdorff
    measure and $B(x,r)$ denotes the ball with center $x$ and radius
    $r$.
  \item[$(\Omega)$] The domain $\Omega$ satisfies a measure density condition,
    that is, $|\Omega \cap B(x,r)| \simeq r^d$ uniformly for all $x \in
    \Omega$ and $r<1$.
  \end{enumerate}
\end{assumption}

In the case of $\fD = \emptyset$,
Assumption~\ref{assu-general}~$(\cN)$ \emph{a fortiori} requires that
$\Omega$ is a (weak) Lipschitz domain, i.e., a Lipschitz manifold
where we have bi-Lipschitz boundary charts locally for the whole
boundary. In the other extreme case $\fD = \partial\Omega$, where each
$D_i =\partial \Omega$, we do not require boundary charts at all, but
merely that $\partial \Omega$ is $(d-1)$-set in the sense of
Jonsson\&Wallin, cf.~\cite[Chapter~II]{Jonsson84}, which is another
notion for condition~\eqref{eq:ahlf} in~$(\fD)$. In the same way, the
measure density condition in condition~$(\Omega)$ means that $\Omega$
is a $d$-set.

\begin{remark}
  \label{rem:geometry}
  This geometric setup that we rely on by means of
  Assumption~\ref{assu-general} is the one
  of~\cite{ABHR14,Egert2018,DtER17}, it gives us a Sobolev extension
  operator and maximal parabolic regularity for the negative Laplacian
  on negative Sobolev spaces by means of its square root. In fact, the
  recent works~\cite{Bechtel2024,Bechtel2020,BBRHT19} of Bechtel,
  Egert and collaborators have generalized the admissible geometric
  setting for this immensely and we could quite directly transfer our
  results to their setting, \emph{however} at the price of not being
  able to consider different Dirichlet boundary parts $D_i$ for each
  component $u_i$ of $\bmu$, cf.~\cite[][Rem.~4.5]{Bechtel2024}. Since
  the less general geometric framework of
  Assumption~\ref{assu-general} should still easily be enough for
  nearly all practical applications of our results, we have decided to
  stick with this one and rather consider possibly different Dirichlet
  boundary parts for each component.
\end{remark}



\begin{definition}[Sobolev spaces] \label{def:sobo-spaces} For $D \subseteq
  \partial\Omega$ closed and for all $q \in [1,\infty)$ we define
  the Sobolev space $W^{1,q}_D(\Omega)$ by
  \begin{equation*}
    W^{1,q}_D(\Omega) \define \overline{\Bigl\{ \psi|_\Omega \colon \psi \in
      C_c^\infty(\R^d) \text{ and } \supp(\psi) \cap D = \emptyset
      \Bigr\}}^{\|\cdot\|_{W^{1,q}(\Omega)}},
  \end{equation*}
  as a subspace of $W^{1,q}(\Omega)$, where the norm is given by
  $\psi \mapsto \bigl( \int_\Omega |\nabla \psi|^q + |\psi|^q
  \bigr)^{1/q}$.
  \begin{enumerate}[(i)]
  \item If $q \in(1,\infty)$, then the anti-dual of
    $W^{1,q}_D(\Omega)$ will be denoted by $W^{-1,q'}_D(\Omega)$,
    where $1/q + 1/q' = 1$, that is, $W^{-1,q}_D(\Omega)$ is the space
    of continuous \emph{anti}linear functionals on
    $W^{1,q'}_D(\Omega)$, extending the $L^2$-inner
    product. 
  \item For $m \geq 1$ a fixed integer, set
    $\bmW^{1,q}_\fD(\Omega) \coloneqq \prod_{i=1}^m W^{1,q}_{D_i}(\Omega)$ with
    \begin{equation*}
      \|\bmu\|_{\bmW^{1,q}_\fD(\Omega)} \define \Bigl(\sum_{i=1}^m
      \|u_i\|_{W^{1,q}_{D_i}(\Omega)}^q\Bigr)^{\frac1q} \qquad (\bmu =
      (u_1,\dots,u_m)),
    \end{equation*}
    and let $\bmW^{-1,q}_\fD(\Omega)$ be the anti-dual of $\bmW^{1,q'}_\fD(\Omega)$.
  \end{enumerate}
\end{definition}

\begin{remark} \label{r-fortsetz} In the geometric setting of
  Assumption~\ref{assu-general}, there exists a continuous linear
  extension operator
  $\bm{E} \colon \bmW^{1,q}_{\fD}(\Omega) \to \bmW^{1,q}(\R^d)$ which
  is universal in $q \in [1,\infty)$, see~\cite[Lemma~3.2]{ABHR14},
  and which simultaneously also extends
  $\bmL^q(\Omega) \to \bmL^q(\R^d)$. This enables the usual Sobolev
  embeddings and Rellich-Kondrachev type compactness properties for
  the Sobolev spaces, although we will not make explicit use of them.
\end{remark}

\begin{remark} \label{rem:neg-spaces} The negative order Sobolev spaces
  turn out to be well suited for the treatment of elliptic and
  parabolic problems if inhomogeneous Neumann boundary conditions are
  present (see e.g.~\cite[Sect.~3.3]{Lio68}), and also if right-hand
  sides of distributional type, e.g.\ surface densities,
  appear~\cite{HR09,HR11}.  In fact, there is often an intrinsic
  connection between (spatial) jumps in the coefficient function and
  the appearance of surface densities as parts of the right hand
  side~\cite[Chapter~1]{Tam79}.  Interestingly, negative Sobolev
  spaces are also adequate for the treatment of control problems
  subject to the previously mentioned problems and have attracted
  significant interest in this area, cf.\
  e.g.~\cite{KPV14,CCK13,KR13,HMRR10,MMR17a,MMR17b,Hoppe2022}.
\end{remark}

\begin{remark}
  In the foregoing definition we have choosen $\bmW^{1,q}_\fD(\Omega)$
  as the product of the spaces $W^{1,q}_{D_i}(\Omega)$, all with the
  same integrability $q$. In principle, we could also have varied
  these with $i$ and put $\bm{q} \simeq (q_1,\dots,q_m)$. Such a
  modification would be straightforward to implement, but would result
  in quite some notational overhead which, we think, rather obfuscates
  the present work, so we have decided not to include it. Moreover, we
  suggest that in nearly all possible applications of the results
  below, one could diminish the $q_i$ to their common minimum and
  still have a working theory. 
\end{remark}

Under Assumption~\ref{assu-general}, we obtain that the spaces
$\bmW^{-1,q}_\fD(\Omega)$ and $\bmW^{1,q}_\fD(\Omega)$ in fact form
complex interpolation scales with respect to $q$. This is proven
in~\cite[Thm.~3.3, Cor.~3.4]{HJKR15}---see
also~\cite{BechtelEgert2019}---for the scalar-valued case and can be
extended to the vector-valued case via component-wise interpolation~\cite[Prop.~I.2.3.3]{Amann95}.

\begin{lemma} \label{l-vectorint} Let
  $q_1,q_2 \in(1,\infty)$ and
  $\theta \in(0,1)$. Then, for
  $\frac{1}{q}=\frac{1-\theta}{q_1} +\frac{\theta}{q_2}$,
  \begin{equation} \label{eq:intedrpolvec}
    \Bigl[\bmW^{1,q_1}_\fD(\Omega),\bmW^{1,q_2}_\fD(\Omega)\Bigr]_\theta =\bmW^{1,q}_\fD(\Omega)
     \text{~~and~~} \Bigl[\bmW^{-1,q_1}_\fD(\Omega),\bmW^{-1,q_2}_\fD(\Omega)\Bigr]_\theta
    =\bmW^{-1,q}_\fD(\Omega).
  \end{equation}
\end{lemma}




\subsection{Elliptic systems}

We turn to basic notions concerning elliptic systems of size $m$ in $d$
dimensions. Recall the definition of the differential operator $\bmL$ in~\eqref{eq:mainEqStrong}.
\begin{definition}[Coefficient tensor]
  Let $m \in \N$ with $m \geq 1$ be fixed. A \emph{coefficient tensor} $\tA$ consists
  of a collection of bounded and measurable functions
  \begin{equation*}
    A^{ij} \colon \Omega \to \C^{d
      \times d} \qquad b^{ij}, c^{ji}  \colon \Omega \to  \C^d, \qquad d^{ij}
    \colon \Omega \to  \C  \qquad (1 \leq i,j\leq m)
  \end{equation*}
  which we assume to be arrayed in a matrix form according to their $(i,j)$ indices,
  \begin{equation*}
    \tA =
    \begin{bmatrix}
      d & c^\top \\ b & A
    \end{bmatrix} \colon \Omega \to \C^{(m+md) \times (m+md)} \simeq \cL(\C^{m(d+1)}).
  \end{equation*}
  Moreover, let
  $\|\tA\|\coloneqq
  \esssup_{x\in\Omega}\|\tA(x)\|_{\cL(\C^{m(d+1)})}$.
  \label{def:coeff-tensor}
\end{definition}

That being said, we
introduce the system operator as follows:

\begin{definition}[System operator and form] \label{d-systopera}
  Let $\tA$ be a coefficient tensor. Then,  for
  $q \in(1,\infty)$ we define the system operator $\bmL
    \colon \bmW^{1,q}_\fD(\Omega) \to \bmW^{-1,q}_\fD(\Omega)$ by
  \begin{equation}\label{eq:defnsystem} \bigl\langle \bmL \bmu, \bm{v} \bigr\rangle \define \int_\Omega    
    \tA
    \begin{bmatrix}
      \bmu \\ \nabla \bmu
    \end{bmatrix} \cdot \overline{\begin{bmatrix}
      \bm{v} \\ \nabla \bm{v}
    \end{bmatrix}} \qquad (\bmu \in \bmW^{1,q}_\fD(\Omega),~\bm{v}\in W^{1,q'}_\fD(\Omega)).
  \end{equation}
\end{definition}

\begin{remark}
  Due to the presupposed boundedness of the coefficient tensor $\tA$ and
  H\"older's inequality it is clear that $\bmL$ as
  in~\eqref{eq:defnsystem} is indeed a well defined continuous and linear operator between
  $\bmW^{1,q}_\fD(\Omega)$ and $\bmW^{-1,q}_{\fD}(\Omega)$ whose norm is
  bounded by $\|\tA\|$. 
  \label{rem:sys-op}
\end{remark}

We next turn to a notion of \emph{ellipticity} for the just defined
system operator. With the wording of \emph{weakly} and \emph{strongly}
elliptic, we follow~\cite[Ch.~3.4]{GM12}.

\begin{definition}[Weakly elliptic, G{\aa}rding's
  inequality] \label{def:elliptic} We say that the operator $\bmL$ satisfies
  \emph{G{\aa}rding's inequality} if, for some $\lambda \geq 0$, the quadratic form
  induced by $\bmL$ is coercive with constant $\gamma>0$ in the following sense:
  \begin{equation} \label{eq:coerc} \Re\blangle \bmL \bmu, \bmu\brangle +
   \lambda \int_\Omega |\bmu|^2 \geq \gamma\, \int_\Omega|\nabla \bmu|^2 \qquad (\bmu \in \bmW^{1,2}_\fD(\Omega)).
 \end{equation}  In this case, we also say that
 $\bmL$ is $\lambda$-\emph{weakly elliptic}. 
 If $\lambda = 0$ can be
 chosen, we say that $\bmL$ 
 is
 \emph{strongly} elliptic.
\end{definition}

\begin{remark} \label{r-suff} We consider the abstract ellipticity condition for
  $\bmL$ in the form of the G{\aa}rding inequality~\eqref{eq:coerc} adequate for
  our context because of its generality. However, the condition is quite
  nontrivial in the sense that it is not immediate at all to give direct (and
  not too strong) conditions on the coefficient tensor $\tA$ which are
  sufficient for~\eqref{eq:coerc}. This already concerns the principal
  part $-\nabla\cdot A\nabla \colon \bmW^{1,q}_\fD(\Omega) \to
  \bmW^{-1,q}_\fD(\Omega)$ of $\bmL$ given by
  \begin{equation}\label{eq:second-order-only}
    \blangle -\nabla\cdot A\nabla\bmu,\bm{v}\brangle
    \define \int_\Omega A\nabla \bmu \cdot \overline{\nabla \bm{v}}, \qquad
    (\bmu \in \bmW^{1,q}_\fD(\Omega),~\bm{v}\in \bmW^{1,q'}_\fD(\Omega)) 
  \end{equation}
  whose ellipticity is sufficient for weak ellipticity of
  the full operator.
  \begin{itemize}
  \item The strongest sufficient condition is the \emph{Legendre
      condition} which postulates that there exists $\gamma > 0$
    such that
    \begin{equation}\label{eq:legendre-bedingung}
      \Re \sum_{i,j=1}^m\fA^{ij}(x) \xi_j \cdot \overline{ \xi_i} \geq \gamma |\xi|_2^2 \qquad
      (\xi \in \C^{d \times m},~\text{almost all}~x\in\Omega),
    \end{equation}
    where $\xi_i$ denotes the $i$-th column of the matrix $\xi$
    and we recall $\fA^{ij}(x) \in \C^{d\times d}$ for
    $x \in \Omega$. It is immediate to verify that this
    condition is even sufficient for $-\nabla\cdot\fA\nabla$ to
    be strongly elliptic.
  \item The \emph{Legendre-Hadamard condition} is derived from
    the Legendre condition by considering only the special class
    of rank-one matrices $\xi = \eta \otimes \zeta$ with
    $\eta \in \R^d$ and $\zeta \in \C^m$
    in~\eqref{eq:legendre-bedingung}. If $\fD = \partial\Omega$
    and the coefficients $A$ are in fact constant, then
    the Legendre-Hadamard condition implies that
    $-\nabla\cdot\fA\nabla$ is strongly elliptic. If $\fA$ is
    uniformly continuous on $\overline\Omega$, then we still
    recover weak ellipticity~(\cite[Ch.~3.4]{GM12}).

    Unfortunately, if $D_i = \emptyset$ for $i=1,\dots,m$, then the Legendre
    Hadamard condition does in general not imply that $-\nabla\cdot\fA\nabla$
    is weakly elliptic, even for constant coefficient tensor $\fA$. (The system
    induced by the Cauchy-Riemann operator is a counterexample; see
    also~\cite{Zha89} or~\cite{Verchota2015}.) We refer
    to~\cite{Zha10,Serre2006} for more details and some positive results
    in this direction; note also that Korn's inequality is an interesting
    particular case~\cite{NP14}.  It seems unknown whether the situation
    improves for mixed boundary conditions where $D_i \neq \emptyset$ for all
    (or some) $i = 1,\dots,m$.

    Note however
    that \emph{if} $-\nabla\cdot\fA\nabla$ is strongly elliptic for
    $\fD = \partial\Omega$, then the Legendre-Hadamard condition
    is indeed satisfied even if $\fA$ is only measurable and
    essentially bounded. This follows from considering, for
    $\eta \in \R^d$ and $\xi \in \C^m$, a family of localized
    sawtooth type functions
    $\bmu_\eps \in \bmW^{1,2}_{\partial\Omega}(\Omega)$ given by
    \begin{equation*}
      \bmu_\eps(x) \define \eps \, \varphi(x) \, w(\eps^{-1} \eta\cdot x)\,
      \xi  \qquad (x \in \Omega,~\eps > 0),
    \end{equation*}
    where $\varphi\in C_c^\infty(\Omega)$ and $w$ is, say, a
    $2$-periodic sawtooth function on the real line induced by
    $t \mapsto 1-|t|$ on $[-1,1]$. Inserting these into the
    G{\aa}rding inequality~\eqref{eq:coerc} with $\lambda = 0$
    and letting $\eps \downarrow 0$ yields the Legendre-Hadamard
    condition.
  \end{itemize}
  We note that the cited works all refer to the (most important)
  case of a \emph{real} coefficient tensor $\tA$ and real
  function spaces. For real $\tA$, we find that
  $\langle \bmL \bmu,\bmu\rangle$ decomposes into
  \begin{equation*}
    \Re \blangle \bmL\bmu,\bmu\brangle = \blangle \bmL(\Re \bmu),\Re
    \bmu)\brangle + \blangle \bmL(\Im \bmu),\Im \bmu\brangle
  \end{equation*}
  with the real and imaginary parts $\Re \bmu$ and
  $\Im \bmu$ of $\bmu \in \bmW^{1,2}_\fD(\Omega)$, both of which
  are of course \emph{real} functions. Thus the theory presented
  in the mentioned references can directly be transferred to the present
  case of complex function spaces if $\tA$ is a real-valued coefficient tensor.
\end{remark}

\begin{remark}\label{r-stable}
  The G{\aa}rding inequality~\eqref{eq:coerc} is stable with respect to small perturbations in the
  following sense: if the operator $\bmL$ induced by the coefficient tensor $\tA$
  is $\lambda$-weakly elliptic with constant $\gamma$, and if $\tB$ is another
  coefficient tensor such that $\|\tB\| < \gamma/2$, then the operator induced by
  $\tA + \tB$ is still $(\lambda+\gamma/2)$-weakly elliptic, with constant $\gamma/2 > 0$.
\end{remark}

Of course, the abstract form of the G{\aa}rding
inequality~\eqref{eq:coerc} combines immediately with the Lax-Milgram
lemma: If $\bmL$ is $\lambda$-weakly elliptic, then
$\bmL + \Lambda \colon \bmW^{1,2}_\fD(\Omega) \to
\bmW^{-1,2}_\fD(\Omega)$ is a topological isomorphism for every
$\Lambda > \lambda$.  Moreover, the norm of the associated inverse is
bounded by $\min(\Lambda-\lambda,\gamma)^{-1}$ with the coercivity
constant $\gamma > 0$ as in~\eqref{eq:coerc}. In fact, we can extend
the isomorphism property to the non-Hilbert spaces in the Sobolev
scale using Sneiberg's extrapolation principle and keep a uniform
bound:

\begin{theorem} \label{t-systemregula} Let $\bmL$ be $\lambda$-weakly elliptic. Then
  for every $\Lambda > \lambda$ there
  is a number $\delta>0$ such that
  \begin{equation} \label{eq:toopsystem} \bmL+\Lambda\colon\bmW^{1,q}_\fD(\Omega) \to \bmW^{-1,q}_\fD(\Omega)
  \end{equation}
  is a topological isomorphism for every $q \in (1,\infty)$ such that
  $|q-2| < \delta$. The number $\delta>0$ can be chosen uniformly with
  respect to all coefficient tensors $\tA$ with a uniform bound on
  $\|\tA\|$ inducing $\lambda$-weakly elliptic system operators with
  the same coercivity constant $\gamma$.
\end{theorem}

\begin{proof}
  It was already mentioned in Remark~\ref{rem:sys-op} that the
  system operator~\eqref{eq:toopsystem} with a coefficient
  tensor $\tA$ as defined in Definition~\ref{d-systopera} is
  continuous for every $q \in(1,\infty)$, and
  that the norms of the respective operators are uniformly
  bounded if $\|\tA\|$ is uniformly bounded.  Since
  $\bmL$ is $\lambda$-weakly elliptic, the
  Lax-Milgram lemma shows that~\eqref{eq:toopsystem} is
  continuously invertible for $q =2$ and the operator norm of
  the inverse
  is bounded by $\min(\Lambda-\lambda,\gamma)^{-1}$.  So the assertions
  follow from $(\bmW^{\pm1,q}_{\fD}(\Omega))_{1<q<\infty}$ being an
  interpolation scale by Lemma~\ref{l-vectorint} and Sneiberg's
  extrapolation principle Theorem~\ref{t-sneib}.
\end{proof}

We note that the foregoing result has been established already
in~\cite[Ch.~6]{HJKR15} by different means; we obtain it here as a
mere byproduct of the G{\aa}rding inquality assumption. One can
consider it as an optimal elliptic regularity result.

Finally, we fix a particular system operator which will serve as an
anchor for interpolation of maximal regularity spaces in the next
section; of course, we are talking about the Laplacian:

\begin{definition}[{Negative Laplacian}] \label{d-2} For 
  $q \in(1,\infty)$, define the continuous linear operator
  $-\bm{\Delta}_q\colon \bmW^{1,q}_\fD(\Omega) \to \bmW^{-1,q}_\fD(\Omega)$
   by
   \begin{equation} \label{eq:formdef0} \bigl\langle -\bm{\Delta}_q  \bmu, \bm{v} \bigr\rangle \define 
     \int_\Omega
     \nabla \bmu\cdot\overline{\nabla\bm{v}} 
     \qquad (\bmu \in
     W^{1,q}_\fD(\Omega),~\bm{v} \in W^{1,q'}_\fD(\Omega)).
  \end{equation}
\end{definition}

Clearly, $-\bm{\Delta}_q$ is the operator associated with the
coefficient tensor $\tA$ made up of $\fA^{ii} = \id_d$ with the
$(d\times d)$-identity matrix $\id_d$, and all other coefficients
zero. It is strongly elliptic, in particular,
Theorem~\ref{t-systemregula} applies for any $\Lambda > 0$ and we
follow up with the next definition (for $\Lambda = 1$):

\begin{definition}[Isomorphism index] \label{d-isoindex} We call a number
  $q \in(1,\infty)$ an \emph{isomorphism index} if
  \begin{equation*}
    -\bm{\Delta}_q + 1 \colon \bmW^{1,q}_{\fD}(\Omega) \to \bmW^{-1,q}_{\fD}(\Omega)
  \end{equation*}
  is a topological isomorphism. We denote by $\fI$ the set of
  isomorphism indices. Although $\fI$ in fact depends on the given
  geometry, we do not keep track of this explicitly since the geometry
  is supposed to be fixed by Assumption~\ref{assu-general}.
\end{definition}

\begin{proposition} \label{p-iso-indices-open} The set $\fI$ is an open interval which
  contains~2.
\end{proposition}
\begin{proof}
  If follows from Lemma~\ref{l-vectorint} that $\fI$ is connected,
  since $\bmW^{1,q}_{\fD}(\Omega)$ and $\bmW^{-1,q}_{\fD}(\Omega)$
  form an interpolation scale with respect to $q$. Moreover,
  $2 \in \fI$ by the Lax--Milgram lemma or
  Theorem~\ref{t-systemregula}, as mentioned before.  Thus, the result
  follows again from Sneiberg's extrapolation principle as in
  Theorem~\ref{t-sneib}.
\end{proof}

We refer to~\cite[Examples~6.6--6.8]{DtER17} for several geometric
constellations for $\Omega$ and $D_i$ and associated isomorphism index
ranges $\fI$. A statement analogous to
Proposition~\ref{p-iso-indices-open} could also be made for the
general operators $\bmL + \Lambda$ as in Theorem~\ref{t-systemregula};
cf.\ also~\cite[Thm.~5.6, Rem.~5.7]{HJKR15}.


\section{Extrapolation of nonautonomous maximal parabolic
  regularity} \label{Snonaut4}

We next turn to another application of the Sneiberg extrapolation
principle, this one being crucial for the treatment of quasilinear
parabolic systems by our approach: nonautonomous maximal parabolic
regularity. It is based on the theory developed in~\cite{DtER17}.

\subsection{Basics on maximal parabolic regularity}

In all what follows, let $T > 0$ and set $J = (0,T)$. We consider an
abstract Banach space $X$ and, associated, another Banach space $D$
such that $D$ is densely and continuously embedded into $X$.

Let us start by introducing the following (standard) definition.

\begin{definition}[Bochner-Sobolev and maximal regularity
  spaces] \label{d-bochner}Let $r \in(1,\infty)$.
  \begin{enumerate}[(i)]
  \item We denote by $L^r(J;X)$ the space of $X$-valued functions $f$
    on $J$ which are Bochner-measurable and for which
    $t \mapsto \|f(t)\|_X \in L^r(J)$, with its natural norm.
  \item Further define
    \begin{equation*}
      W^{1,r}(J;X)\define \Bigl\{u \in L^r(J;X) \colon u' \in L^r(J;X)\Bigr\},
    \end{equation*}
    equipped with the sum norm. Here, $u'$ is to be understood as the
    (time) derivative of $u$ in the sense of $X$-valued distributions,
    cf.~\cite[Sect.~III.1]{Amann95}. 
  \item We write
    \begin{equation*}
      \bW^{1,r}(J;D,X) \define  L^r(J;D) \cap W^{1,r}(J;X)
    \end{equation*}
    for the space of maximal parabolic regularity, and we introduce the
    (closed) subspace, cf.\ Lemma~\ref{lem:extension},
    \begin{equation*}
      \bW^{1,r}_0(J;D,X)\define \Bigl\{ u \in \bW^{1,r}(J;D,X) \colon
      u(0)=0\Bigr\}.\end{equation*}
    Equipped with the sum norms, both
    spaces are complete.
  \end{enumerate}
\end{definition}

In this paper we consider the following notion of maximal parabolic
regularity in the nonautonomous setting in the constant domain case.

\begin{definition}[Maximal parabolic regularity (nonautonomous)]
  \label{d-maxpar} Let $(\cA(t)_{t \in J}$ be an operator family of
  closed operators on $X$ with common domain $D$ such that
  $J \ni t \mapsto \cA(t) \in \cL(D \to X)$ is bounded and strongly
  measurable.  Then, for $r \in (1,\infty)$, we say that 
  $\cA$
  satisfies \emph{(nonautonomous) maximal parabolic
    $L^r(J;X)$-regularity}, if for any $f \in L^r(J;X)$ there is a
  unique function $u \in \bW^{1,r}(J;D,X)$ which satisfies
  \begin{equation} \label{eq:0paragleich} u'(t) +\cA(t)u(t)=f(t) \quad \text{in}~X,
    \qquad u(0) = 0,
  \end{equation}
  for almost all $t \in J$. 
  Equivalently, the (bounded and linear) parabolic operator
  $\partial + \cA \colon \bW^{1,r}_0(J;D,X) \to L^r(J;X)$, defined
  by
  $\bigl[(\partial + \cA)u\bigr](t) \define u'(t) + \cA(t)u(t)$,
  is continuously invertible.
\end{definition}
\begin{remark} \label{r-invariantr} In the autonomous case of
  Definition~\ref{d-maxpar}, that is, if all operators $\cA(t)$ are in
  fact equal to one (fixed) operator $\cA_0$, and $\cA$ satisfies
  maximal parabolic $L^r(J;X)$-regularity for \emph{some}
  $r \in (1,\infty)$, then $\cA$ satisfies maximal parabolic
  $L^s(I;X)$-regularity for \emph{every} $s \in(1,\infty)$ and all
  (finite) intervals $I\subset \R$, and $-\cA_0$ is the generator of
  an analytic semigroup on $X$. These are classical results by
  Dore~\cite{Dor93}. Accordingly, we say that $\cA_0$ satisfies
  \emph{maximal parabolic regularity on $X$} in this case.
\end{remark}

Since we consider a \emph{finite} time interval $J$, the
property of maximal regularity is moreover invariant under
scalar shifts:

\begin{corollary}
  \label{cor:mpr-shift}
  In the setting of Definition~\ref{d-maxpar}, if the family
  $(\cA(t))_{t\in J}$ satisfies nonautonomous maximal parabolic
  $L^r(J;X)$-regularity, then so does $(\cA(t)+ \mu)_{t\in J}$ for any scalar
  $\mu\in\C$. Moreover, the operator norms of
  \begin{equation*}
    \bigl(\partial+\cA+\mu\bigl)^{-1} \colon L^r(J;X) \to \bW^{1,r}_0(J;D,X) 
  \end{equation*}
  are uniformly bounded in $|\mu|$ and the operator norm of $(\partial+\cA)^{-1}$.
\end{corollary}

\begin{proof}
  We use that for every $f \in L^r(J;X)$, the solutions $u,v \in \bW^{1,r}_0(J;D,X)$ to
    \begin{align*}
      u'(t) + \cA(t)u(t) + \mu u(t) & = f(t)\phantom{e^{\mu t}} \quad \text{in } X,
      \qquad u(0) = 0,
      \intertext{and}
      v'(t) + \cA(t)v(t) & = e^{\mu t}f(t) \quad \text{in } X,
      \qquad v(0) = 0
  \end{align*}
  can be transformed into each other along $u(t) \define e^{-\mu t}v(t)$. In
  particular, one is unique if and only if the other is so. Finally, from the
  assumption, we have
  \begin{equation*}
    \|v\|_{\bW^{1,r}_0(J;D,X)} 
    \leq  \bigl\|(\partial +
    \cA)^{-1}\bigr\|_{L^r(0,T;X) \to \bW^{1,r}_0(0,T;D,X)} \|e^{\mu \cdot}f\|_{L^r(J;X)}.
  \end{equation*}
  and with $u(t) \define e^{-\mu t}v(t)$ this leads to
  \begin{equation*}
    \|u\|_{\bW^{1,r}_0(J;D,X)} \leq e^{|\Re \mu|T}
    \bigl(|\mu|\, C_{D\embeds X} + 1\bigr) \, \bigl\|(\partial +
    \cA)^{-1}\bigr\|_{L^r(0,T;X) \to \bW^{1,r}_0(0,T;D,X)} \|f\|_{L^r(J;X)}
  \end{equation*}
  with the norm $C_{D\embeds X}$ of the embedding $D \embeds
  X$. This implies the claims.
\end{proof}

We take the opportunity to state another lemma which
collects some classical embedding results relating the maximal parabolic
regularity spaces to their trace space $(X,D)_{1-1/r,r}$, a real
interpolation space between $X$ and $D$. For proofs, we refer
to~\cite[Ch.~III.4.10]{Amann95}; see
also~\cite[Lem.~2.11]{DtER15} for a particularly simple proof
of~\eqref{eq:maxreghoelderembed}.
  
\begin{lemma} \label{l-maxparregfacts} Let $r \in(1,\infty)$. Then we have the
    embeddings
    \begin{equation} \label{eq:embedcont} \bW^{1,r}(J;D,X) \embeds
      C\bigl(\overline{J}; (X, D)_{1-1/r,r}\bigr)
    \end{equation}
    and, for $0<\zeta <1-\frac1r$ and
    $0 < \alpha = 1 - \frac {1}{r}-\zeta$:
    \begin{equation} \label{eq:maxreghoelderembed} \bW^{1,r}(J;D,X)
      \embeds C^\alpha\bigl(\overline J;(X,D)_{\zeta,1}\bigr).
    \end{equation}
\end{lemma}

In fact, the interpolation space in
embedding~\eqref{eq:embedcont} is exactly the correct continuous
trace space in the maximal regularity setting. This can be
inferred for instance from the following consequence
of~\cite[Prop.~2.1]{Ama05a} and its proof, which is useful in a number of
ways:

\begin{lemma}
  \label{lem:inhomo-solution}
  Let $r \in (1,\infty)$ and let
  $(\cA(t)_{t \in J}$ be an operator family satisfying
  nonautonomous maximal parabolic $L^r(J;X)$-regularity with
  common domain $D$. Let $\alpha$ be an upper bound for the
  operator norm of $(\partial + \cA)^{-1}$. Suppose there exists
  another fixed operator satisfying maximal parabolic regularity
  on $X$ with domain $D$. Then for every
  $u_0 \in (X,D)_{1-1/r,r}$ and every $f \in L^r(J;X)$ there
  exists a unique function $u \in \bW^{1,r}(J;D,X)$ which satisfies
  \begin{equation*}
    u'(t) +\cA(t)u(t)=f(t) \quad \text{in}~X, \qquad u(0) = u_0
  \end{equation*}
  and we have the estimate
  \begin{equation*}
    \|u\|_{\bW^{1,r}(0,T;D,X)} \lesssim(\alpha+1) \Bigl[\|f\|_{L^r(0,T;X)} + \|u_0\|_{(X,D)_{1-1/r,r}}\Bigr].
  \end{equation*}    
\end{lemma}

\begin{remark}
  By Lemma~\ref{lem:inhomo-solution}, whenever there exists an
  operator satisfying maximal parabolic regularity on $X$ with
  domain $D$, then for every given $u_0 \in (X,D)_{1-1/r,r}$,
  Lemma~\ref{lem:inhomo-solution} yields a function
  $\overline{u_0} \in \bW^{1,r}(J;D,X)$ with
  $\overline{u_0}(0) = u_0$ and the estimate
  \begin{equation*}
    \|\overline{u_0}\|_{\bW^{1,r}(0,T;D,X)} \lesssim \|u_0\|_{(X,D)_{1-1/r,r}}
  \end{equation*}
  which is the converse
  to~\eqref{eq:embedcont}.\label{rem:trace-extend}  
\end{remark}

We can moreover
use the foregoing lemma to easily derive an extension operator acting between
maximal regularity spaces on different time intervals:


\begin{lemma}[{\cite[Lemma~7.2]{Ama05b}}]
  \label{lem:extension}
  Let $r \in(1,\infty)$ and let $\cA$ be an
  operator satisfying maximal parabolic regularity on $X$ with
  domain $D$.  Then, for every $S \in J$ and every
  $u \in \bW^{1,r}(0,S;D,X)$, the unique solution $v\enifed E_S u$
  to
  \begin{equation*}
    v'(t) + \cA v(t) = f(t) \quad \text{in } X, \qquad v(0) =
    u(0), \qquad  f(t) \define \chi_{(0,S)}(t) \bigl[u'(t) + \cA u(t)\bigr]
  \end{equation*}
  induces a continuous linear extension operator
  \begin{equation*}
    E_S \colon \bW^{1,r}(0,S;D,X) \to \bW^{1,r}(0,T;D,X), \qquad
    (E_Su)|_{(0,S)} = u,
  \end{equation*}
  with
  \begin{equation*}
    \|E_Su\|_{\bW^{1,r}(0,T;D,X)} \lesssim
    \|u\|_{\bW^{1,r}(0,S;D,X)} + \|u(0)\|_{(X,D)_{1-1/r,r}},
  \end{equation*}
  where the implicit constant does \emph{not} depend on $S$.
\end{lemma}

\begin{remark}
  In Lemmata~\ref{lem:inhomo-solution} and~\ref{lem:extension}, it would have
  been sufficient to assume that there exists an operator $\cA$ which is the
  negative generator of an analytic semigroup on $X$ with domain $D$ instead of
  requiring maximal parabolic regularity for such an operator. Since any
  operator satisfying maximal parabolic regularity is also the negative
  generator of an analytic semigroup~\cite{Dor93}, we have posed a stronger
  condition for the sake of simplicity--but then, in fact, the natural operator
  candidate for either condition for $X = \bmW^{-1,q}_\fD(\Omega)$ as we
  consider later would have been the same: the negative Laplacian.
\end{remark}

As a last preliminary lemma on nonautonomous maximal parabolic
regularity, we quote~\cite[Lem.~4.1]{Ama04a}
and~\cite[][Rem.~3.2]{Ama05} which will allow us to deal with
restrictions and shifts in the context of nonautonomous maximal parabolic regularity in a uniform manner:

\begin{lemma}
  \label{lem:maxreg-shifts}
  Adopt the setting of Lemma~\ref{lem:inhomo-solution} and assume
  in addition that for any subinterval
  $(\tau,S) \subseteq J$, 
  $\cA$ satisfies nonautonomous maximal parabolic $L^r(\tau,S;X)$-regularity
  with common domain $D$. Then for every $u_0 \in (X,D)_{1-1/r,r}$ and every
  $f \in L^r(\tau,S;X)$ there exists a unique function
  $u \in \bW^{1,r}(\tau,S;D,X)$ which satisfies
  \begin{equation*}
    u'(t) +\cA(t)u(t)=f(t) \quad \text{in}~X, \qquad u(\tau) = u_0
  \end{equation*}
  for almost all $t \in(\tau,S)$, and we have
  the estimate
  \begin{equation*}
    \|u\|_{\bW^{1,r}(\tau,S;D,X)} \lesssim 
    (\alpha+1)
    \Bigl[\|f\|_{L^r(\tau,S;X)} + \|u_0\|_{(X,D)_{1-1/r,r}}\Bigr]
  \end{equation*}
  where the implicit constant does \emph{not} depend on $\tau$
  and $S$.
\end{lemma}

We finally note that---of course---the (negative shifted) Laplacian
$-\bm{\Delta}_q + 1$ as in Definition~\ref{d-2} satisfies maximal parabolic
regularity on $\bmW^{-1,q}_{\fD}(\Omega))$. This will enable us to use it as a
reference operator for both extrapolation of maximal parabolic regularity and
the foregoing results, cf.\ Remark~\ref{rem:trace-extend} and
Lemma~\ref{lem:maxreg-shifts}.
Since we have
defined $-\bm{\Delta}_q + 1$ as a bounded linear operator
$\bmW^{1,q}_\fD(\Omega) \to \bmW^{-1,q}_\fD(\Omega)$, we need the associated parameters $q$ to be
isomorphism indices here, recall Definition~\ref{d-isoindex}. Indeed, the following result is to be found
e.g.\ in~\cite[Lem.~6.9 \& Thm.~6.10]{DtER17} for the scalar case, but it
transfers to the system case immediately since the Laplacian is a
diagonal operator:

\begin{theorem} \label{thm:laplace-max-reg} If $q \in \fI$, then
  $-\bm{\Delta_}q + 1$ satisfies maximal parabolic regularity on
  the space $\bmW^{-1,q}_\fD(\Omega)$ with domain
  $\bmW^{1,q}_\fD(\Omega)$.
\end{theorem}

\subsection{Nonautonomous maximal parabolic regularity for elliptic systems}

We now move to nonautonomous versions of the elliptic system operators $\bmL$. To
this end, we extend the notion of a coefficient tensor:

\begin{definition}[Coefficient tensor family]\label{def:coeff-tensor-family}
  A family $(\tA_t)_{t\in J}$ is a \emph{coefficient tensor family} when
  $\tA_t$ is a coefficient tensor in the sense of
  Definition~\ref{def:coeff-tensor} for each $t\in J$, such that
  \begin{equation*}
    \|\tA\|_\infty \define \sup_{t\in J} \|\tA_t\| < \infty 
  \end{equation*}
  and $J \ni t \mapsto \tA_t \in L^1(\Omega;\cL(\C^{m(d+1)})$ is measurable.
\end{definition}

With a coefficient tensor family we also associate the corresponding family of
system operators $(\bmL_t)_{t\in J}$ as in
Definition~\ref{d-systopera}. See~\cite[Ch.~7.1]{DtER17} regarding the
$L^1$-measurability requirement in the foregoing definition; it is a suitable
minimal requirement to have measurability of
$t \mapsto \langle \bmL_t \bmu, \bm{v}\rangle$ for all
$\bmu,\bm{v} \in \bmW^{1,2}_\fD(\Omega)$.

We are interested in nonautonomous maximal parabolic
$L^r(J;\bmW^{-1,q}_\fD(\Omega))$-regularity for the nonautonomous operator
family $\bmL$. To this end, we leverage a fundamental ``parabolic Lax-Milgram
lemma'' of J.-L.~Lions~(\cite[Section~XVIII.§3]{DL92}) in the Hilbert
space setting $r=q=2$:

\begin{theorem}[Hilbert space case] \label{t-lions} Let
  $(\tA_t)_{t\in J}$ be a coefficient tensor family and $\bmL_t$ the
  associated system operators. Assume that there are
  $\gamma > 0$ and $\lambda \geq 0$ for which the \emph{uniform}
  $\lambda$-weak ellipticity condition is satisfied:
  \begin{equation} \label{eq:ellippt} \inf_{t\in J}\Re \blangle \bmL_t \bmu, \bmu \brangle +
    \lambda \int_\Omega |\bmu|^2 \ge \gamma
    \int_\Omega |\nabla \bmu|^2 \qquad (\bmu \in \bmW^{1,2}_\fD(\Omega))
  \end{equation}
  Then for every $\Lambda > \lambda$, the operator family
  $(\bmL_t + \Lambda)_{t\in J}$ satisfies nonautonomous maximal parabolic
  $L^2(J;\bmW^{-1,2}_\fD(\Omega))$-regularity. More precisely, for every
  $\bm{f} \in L^2(J;\bmW^{-1,2}_\fD(\Omega))$, there exists a unique function
  $\bmu \in \bW^{1,2}_0(J;\bmW^{1,2}_\fD(\Omega)),\bmW^{-1,2}_\fD(\Omega))$
  such that
  \begin{equation*}
    \bmu'(t) + \bmL_t \bmu(t) + \Lambda \bmu(t) = \bm{f}(t) \quad \text{in}~\bmW^{-1,2}_\fD(\Omega),
  \end{equation*}
  for almost all $t \in J$, and we have the estimate
  \begin{equation*}
    \|\bmu\|_{\bW^{1,2}_0(J;\bmW^{1,2}_\fD(\Omega),\bmW^{-1,2}_\fD(\Omega))} \leq
    \frac{\min(\Lambda-\lambda,\gamma) + \|\tA\|_\infty + \Lambda}{\min(\Lambda-\lambda,\gamma)} \,\|\bm{f}\|_{L^2(J;\bmW^{-1,2}_\fD(\Omega))}.
  \end{equation*}
\end{theorem}

Recall that for maximal parabolic regularity, we do not care particularly much
about the shift by $\Lambda$ due to Lemma~\ref{lem:maxreg-shifts}, so the
reader may substitute $\Lambda = \lambda + 1$ in their head.

\begin{remark}[Real solutions]
  \label{rem:real-data}
  If the data $(\tA_t)_{t\in J}$ and $\bm{f}$ in Theorem~\ref{t-lions} are in
  fact \emph{real}, then the solution $u$ will also be a
  real-valued function. Indeed, the theorem holds as well when considering the
  subspace $\R$-$\bmW^{1,2}_\fD(\Omega)$ of real-valued functions of
  $\bmW^{1,2}_\fD(\Omega)$ as a real vector space and deriving
  $\R$-$\bmW^{-1,2}_\fD(\Omega)$ from this one. (We refer
  to e.g.~\cite[][Section~3]{Amann95} for the precise
  constructions.) See also Remark~\ref{r-suff}.
  Since the following
  Theorems~\ref{t-extrapolatsystem},~\ref{t-quasilin},~\ref{thm:quasilinear-2}
  and~\ref{t-quasilin-global} all build upon
  Theorem~\ref{t-lions}, this remark also applies to all of them
  \emph{mutatis mutandis}.
\end{remark}

Of course, the fact that we get a uniform estimate for the solution in
Theorem~\ref{t-lions} is very nice.  Accordingly, let us denote the set of all
coefficient tensor families with given uniform bounds as follows:

\begin{definition}[Uniform coefficient tensor family]\label{def:coeff-tensor-family-uniform}
  Let $\Xi(\lambda,\gamma,M)$ denote the set of all coefficient tensor families
  $(\tA_t)_{t \in J}$ whose associated system operator family
  $(\bmL_t)_{t\in J}$ is uniformly $\lambda$-weakly elliptic with coercivity
  constant $\gamma > 0$ as in~\eqref{eq:ellippt}, and for which
  $\|\tA\|_\infty \leq M$.
\end{definition}

We next upgrade the Theorem~\ref{t-lions} with Sneiberg's extrapolation
principle to the Banach space setting $r,q \neq 2$, which is one of the main
results. Remarkably, there are no additional assumptions.

\begin{theorem} \label{t-extrapolatsystem} Let parameters
  $(\lambda,\gamma,M)$ be given. Then there exist open connected
  intervals $\cI_{t}$ and $\cI_{x}$ with
  $(2,2) \in \cI_{t} \times \cI_{x}$, such that for every $(r,q)
  \in \cI_t \times \cI_x$ we have:
  \begin{enumerate}[(i)]
  \item For every $(\tA_t)_{t\in J}\in \Xi(\lambda,\gamma,M)$, the associated elliptic
    system family $(\bmL_t + \Lambda)_{t \in J}$ satisfies nonautonomous
    maximal $L^r(J;\bmW^{-1,q}_\fD(\Omega))$-regularity for every $\Lambda \in \C$ with
    the common domain $\bmW^{1,q}_\fD(\Omega))$.\label{t-extrapolatsystem-max-reg}
  \item  There is an upper bound on the operator norms of
    \begin{equation*}
      (\partial + \bmL + \Lambda)^{-1} \colon
      L^r(J;\bmW^{-1,q}_\fD(\Omega)) \to \bW^{1,r}_0(J;\bmW^{1,q}_\fD(\Omega),\bmW^{-1,q}_\fD(\Omega))
    \end{equation*}
    which is \emph{uniform} in $(\tA_t)_{t\in J} \in \Xi(\lambda,\gamma,M)$ and $|\Lambda
     - \lambda|$.
  \end{enumerate}
  If the family $(\tA_t)_{t\in J}$ is constant, then $\cI_x$ can be
  chosen such that $\cI_{t} = (1,\infty)$
  in~\ref{t-extrapolatsystem-max-reg}.
\end{theorem}

Note that we do \emph{not} claim a uniform bound on the operator norm
of the inverse parabolic operator in the case when $(\tA_t)$ is
constant and $\cI_t = (1,\infty)$. We suggest
that such a
property could be established for any compactly contained subinterval
of $(1,\infty)$ on the basis
of~\cite[][Thm.~1.3.6]{AshyralyevSobolevskii1994} and careful renorming
arguments.

\begin{proof}[Proof of Theorem~\ref{t-extrapolatsystem}]
  First of all,
  the linear mapping
  \begin{equation}
    \partial + \bmL + \Lambda\colon
    \bW^{1,r}_0(J;\bmW^{1,q}_\fD(\Omega),\bmW^{-1,q}_\fD(\Omega)) \to
    L^r(J;\bmW^{-1,q}_\fD(\Omega)) \label{eq:para-op-maxreg-space}
  \end{equation}
  is bounded for all $r,q \in(1,\infty)$, and its norm can be estimated by
  $1 + \|\tA\|_\infty + \Lambda$, independently of the
  integrability pair $(r,q)$. Moreover, for the integrability
  pair $(r,q) = (2,2)$ and $\Lambda > \lambda$, the
  operator~\eqref{eq:para-op-maxreg-space} is even continuously
  invertible as established in Theorem~\ref{t-lions}.

  Now choose $\delta > 0$ such that the assertion of
  Theorem~\ref{t-systemregula} is satisfied,
  and such that $(2-\delta,2+\delta) \subseteq \fI$. (Recall
  Definition~\ref{d-isoindex}.) Diminishing $\delta$ just slightly and
  invoking Theorem~\ref{thm:laplace-max-reg},
  \begin{equation*}
    \partial - \bm{\Delta}_q + 1 \colon
    \bW^{1,r}_0(J;\bmW^{1,q}_\fD(\Omega),\bmW^{-1,q}_\fD(\Omega)) \to L^r(J;\bmW^{-1,q}_\fD(\Omega)) 
  \end{equation*}
  and
  \begin{equation*}
    \bmL_t + \Lambda \colon \bmW^{1,q}_\fD(\Omega)
    \to \bmW^{-1,q}_\fD(\Omega) \qquad (t \in J)
  \end{equation*}
  are topological isomorphisms for
  $q \in \cI_\delta \define [2-\delta,2+\delta]$ and any
  $r \in(1,\infty)$. Thus $\bmW^{\pm1,q+\delta}_\fD(\Omega)$ and
  $\bmW^{\pm1,q-\delta}_\fD(\Omega)$ each form a pair of common maximal
  parabolic regularity and the maximal regularity spaces---with
  an equivalent norm induced by the reference operators
  $\partial-\bm{\Delta}_{q\pm\delta} + 1$---can be
  interpolated for
  $(r,q) \in (1,\infty) \times \cI_\delta$,
  see~\cite[Ch.~3]{DtER17}.  Of course we next want to use the
  result of Sneiberg as in Theorem~\ref{t-sneib} to extrapolate
  the isomorphism property for
  $\partial + \bmL + \Lambda$ as
  in~\eqref{eq:para-op-maxreg-space} from the integrability pair
  $(2,2)$ to neighboring pairs $(r,q)$. An na\"{i}ve application
  of Sneiberg's extrapolation principle would yield the desired
  result along a line through $(\frac12,\frac12)$ in the $(\frac1r,\frac1q)$-plane,
  which is not sufficient for our means. We thus proceed as
  follows, using the quantitative bounds in
  Theorem~\ref{t-sneib} 
  and the 
  norm estimate for the
  operators~\eqref{eq:para-op-maxreg-space}:

  First, we extrapolate in time, so with respect to $r$, only, and
  obtain an open connected interval $\cI_t$ containing $2$ such that
  the operators~\eqref{eq:para-op-maxreg-space} are isomorphisms for
  integrability pairs $(r,2)$ with $r \in \cI_t$, with a uniform bound
  on their inverses. Next, for each fixed $r \in \cI_t$, we
  extrapolate from $(r,2)$ in the spatial integrability scale with
  $2\pm\delta$ being the outer anchors. By construction of $\cI_t$,
  and the 
  norm bounds on the
  operators~\eqref{eq:para-op-maxreg-space}, the isomorphism ``extrapolation
  range'' with respect to $q$ around each such $(r,2)$ in
  Theorem~\ref{t-sneib} 
  is also uniform in
  $r \in \cI_t$. The assertion follows. 

  The above proof is rather descriptive, and the devil is
  in the details. We refer to~\cite[Thm.~7.3]{DtER17},
  where those details are worked out explicitly in the case of scalar
  parabolic equations. The arguments transfer quite literally to the
  present case of systems using $-\bm{\Delta}_q+1$ as the reference
  operator, as explained above, so we do not repeat them here but
  content ourselves with the foregoing abstract description.

  Now, if the family $\tA$ is in fact \emph{constant}, then we need not
  extrapolate in the time integrability $r$. In fact, in this case, we
  first extrapolate in the space integrability $q$, using
  $q \pm \delta$ as outer anchors. This yields an open connected
  interval $\cI_x$ containing $2$ for which we have the maximal regularity
  property for~\eqref{eq:para-op-maxreg-space} for all integrability pairs
  $(2,q)$ for $q \in \cI_x$. But since $\tA$ was assumed to be
  constant, this implies maximal parabolic regularity
  of~\eqref{eq:para-op-maxreg-space} for all integrability pairs
  $(r,q) \in (1,\infty) \times \cI_x$, see
  Remark~\ref{r-invariantr}.

  In any case, having established the maximal parabolic regularity property for
  $\bmL + \Lambda$ for $\Lambda > \lambda$ with the corresponding uniform
  norm estimates, we obtain the same for
  $\bmL + \Lambda$ for \emph{any} scalar $\Lambda \in \C$ by
  virtue of Corollary~\ref{cor:mpr-shift} and the associated uniform estimate.
\end{proof}

\begin{remark}
  \label{rem:extrapolation-always-iso}
  Let us point out that from the construction in the foregoing proof,
  we always have $\cI_x \subseteq \fI$, the set of isomorphism indices
  for the Laplacian, cf.\ Definition~\ref{d-isoindex}, and
  $-\bm{\Delta}_q + 1$ satisfies maximal parabolic regularity on
  $\bmW^{-1,q}_\fD(\Omega)$ with domain $\bmW^{1,q}_\fD(\Omega)$ for
  all $q \in \cI_x$.
\end{remark}

\section{Wellposedness for quasilinear parabolic systems}\label{Snonaut6}

We next consider quasilinear parabolic systems for $\bmu = (u_1,\dots,u_m)$ as
in~\eqref{eq:mainEqQL}. These shall arise from the nonautonomous differential
operators of the type $\bmL$ that is, we assume that for every $\bmu$,
$(\tA(\bmu)_t)_{t\in J}$ is a coefficient tensor family in the sense of
Definition~\ref{def:coeff-tensor-family}, and the associated system operators
as in Definition~\ref{d-systopera} are given by
$\bmL(\bmu) = (L^1(\bmu),\dots,L^m(\bmu))$ with
\begin{multline*}
  \bigl[\bmL(\bmu)^i_t\bigr]\bmu(t) \coloneqq - \sum_{j=1}^m \divv\bigl(\fA^{ij}(\bmu)_t \nabla u_j(t) +
  b^{ij}(\bmu)_t u_j(t)\bigr) \\+\sum_{j=1}^m
  c^{ij}(\bmu)_t\cdot\nabla u_j(t)  + \sum_{j=1}^md^{ij}(\bmu)_tu_j(t).
\end{multline*}


Then the system under consideration is given by
\begin{equation} \label{eq:system}
  \bmu'(t) + \bmL(\bmu)_t\bmu(t)  =
  \Phi(\bmu)(t) \quad \text{in } X, \qquad \bmu(0) = \bmu_0
\end{equation}
for almost all $t \in J$, with a nonlinear but subordinated function
$\Phi$. Of course we want to consider $X=\bmW^{-1,q}_\fD(\Omega)$ and pursue a
maximal parabolic regularity ansatz for this problem, basing on
Theorem~\ref{t-extrapolatsystem} from Section~\ref{Snonaut4} in conjunction
with a fundamental theorem by Amann (Theorem~\ref{t-Amann} below). It will also
become very handy that Theorem~\ref{t-extrapolatsystem} also offers uniform
estimates for the parabolic solution operators for coefficient tensor families
with uniform data.

The critical feature of the above problem formulation will be that we will
allow $\tA$, so also $\bmL$, and $\Phi$ to operate between function spaces on the time-space
cylinder $J \times \Omega$, that is, we allow for \emph{nonlocal-in-time}
operations. (This is also indicated subtly by writing $\bmL(\bmu)_t$ and
$\Phi(\bmu)(t)$ instead of $\bmL(\bmu(t))$ and $\Phi(\bmu(t))$.) In this
context, we will rely on the \emph{Volterra property} for determinism; for an operator $\Psi$
mapping time-dependent functions on $(0,T)$ into time-dependent functions on
$(0,T)$ (with possibly different range spaces), the Volterra property requires
that for any $S \in (0,T]$,
\begin{equation*}
  \Bigl[u(t) = v(t)~\text{f.a.a.}~t \in (0,S)\Bigr] \quad
  \implies \quad \Bigl[\Psi(u)(t) = \Psi(v)(t)~\text{f.a.a.}~t \in (0,S)\Bigr].
\end{equation*}

Now, why care about nonlocal-in-time nonlinear functions $\fA$ and
$\Phi$?  Clearly, such a phenomenon occurs in many interesting
applications, see for example~\cite{HLY19,Ama05,Walker2023} and the
references therein, and this alone would make it worthwhile to
consider. But there is another point of view, at first more subtle,
but in the end rather striking: It allows to cast also parabolic
systems whose stationary counterpart (at initial time) \emph{fails to
  be elliptic} in the sense of Definition~\ref{def:elliptic} in a form
which makes them susceptible to our analysis. To make this plausible,
let us briefly consider a second-order only parabolic
system~\eqref{eq:system} (cf.~\eqref{eq:second-order-only}) in three
functions $\bmu = (u_1,u_2,u_3)$ governed by the tensor
\begin{equation*}
  \widetilde A(\bmu) \define
  \begin{pmatrix}
    A^{1,1}(\bmu) &  A^{1,2}(\bmu) &  A^{1,3}(\bmu) \\
    A^{2,1}(\bmu) &  A^{2,2}(\bmu) &  A^{2,3}(\bmu) \\
    0 & 0 & 1
  \end{pmatrix}
\end{equation*}
with the initial values $\bmu(0) = \bmu_0$. Suppose that the nonlinear
functions act locally-in-time, i.e.,
$\widetilde A(\bmu)_t \simeq \widetilde A(\bmu(t))$. Now, it can happen
that---depending on the magnitude and sign of the coefficient functions---the
tensor at time $t=0$
\begin{equation*}
  \widetilde A(\bmu)\bigr|_{t=0} =
  \begin{pmatrix}
    A^{1,1}(\bmu_0) &  A^{1,2}(\bmu_0) &  A^{1,3}(\bmu_0) \\
    A^{2,1}(\bmu_0) &  A^{2,2}(\bmu_0) &  A^{2,3}(\bmu_0) \\
    0 & 0 & 1
  \end{pmatrix}
\end{equation*}
is \emph{not} elliptic in the sense of Definition~\ref{def:elliptic},
but the \emph{reduced} tensor related to $(u_1,u_2)$ only,
\begin{equation*}
  A(u_1,u_2)\bigr|_{t=0} \define
  \begin{pmatrix}
    A^{1,1}(\bmu_0) &  A^{1,2}(\bmu_0)  \\
    A^{2,1}(\bmu_0) &  A^{2,2}(\bmu_0) 
  \end{pmatrix}
\end{equation*}
\emph{is} so.  This phenomenon occurs for example in particular
realizations of the famous Keller-Segel model in chemotaxis, see
Section~\ref{sec:example} below. But suppose we can solve the
(time-dependent!)  equation for $u_3$ in dependence on the other
functions $(u_1,u_2)$ living on $J\times\Omega$ with a suitably
well defined solution operator $\cS \colon (u_1,u_2) \to
u_3$. Then we can recast the original non-elliptic problem
equivalently using the reduced tensor
\begin{equation*}
  A(u_1,u_2) \define
  \begin{pmatrix}
    A^{1,1}(u_1,u_2,\cS(u_1,u_2)) &  A^{1,2}(u_1,u_2,\cS(u_1,u_2))  \\
    A^{2,1}(u_1,u_2,\cS(u_1,u_2)) &  A^{2,2}(u_1,u_2,\cS(u_1,u_2)) 
  \end{pmatrix}
\end{equation*}
to obtain a problem which is (at the initial time) elliptic and will thus turn
out to be tractable by our results. The eliminated component $u_3$ can then be
recovered using $\cS$. However, the coefficient functions
$(u_1,u_2) \mapsto A^{i,j}(u_1,u_2,\cS(u_1,u_2))$ are \emph{nonlocal-in-time} since
$\cS$ is so, even when the $A^{i,j}$ were not! This observation shall serve as
another motivation to consider nonlinear system operators with a
nonlocal-in-time dependence on the functions. Of course, an analogous
modification has to be introduced into the right-hand sides $\Phi$ which is why
we also assume these to act nonlocally-in-time.

It is also clear that the previous explanations do not depend on the exact
format with three equations, of which we remove one, at all. More generally one
can have a system of $m+k$ equations and we implicitly solve $k$ of them,
resulting in a modified system of $m$ equations; this may in particular include
the most convenient case $m=1$ for which one even obtains just a scalar
equation from this reduction, albeit a complicated one. However, this is not
always possible, which is why we establish our results for \emph{systems}. An
example will be given in Section~\ref{sec:example} after the main theorems.

\subsection{Existence and uniqueness results}

The pioneering theorem which allow us to carry out for the foregoing
agenda was established by Amann~(\cite[Thm.~2.1]{Ama05b}) and gives a
comprehensive framework for solving abstract quasilinear problems, in
particular including nonlocal-in-time operations.

\begin{theorem}[Amann] \label{t-Amann} Let
  $r \in(1,\infty)$ and suppose that $D$ and $X$ are
  Banach spaces with dense embedding $D \embeds X$ such that
  there exists an operator satisfying maximal parabolic
  regularity on $X$ with domain $D$. Consider the abstract
  problem
  \begin{equation}\label{eq:amann-abstract}
    u'(t) +\bigl[\cA(u)(t)\bigr]u(t) =F(u)(t) \quad \text{in } X, \qquad u(0)=u_0,
  \end{equation}
  and assume the following:
  \begin{enumerate}[(1)]
  \item \label{t-Amann-evil} The operator $\cA$ maps $\bW^{1,r}(J;D,X)$
    into $L^\infty(J;\cL(D \to X))$ and it is Lipschitz continuous on bounded
    subsets of $\bW^{1,r}(J;D,X)$.
  \item For each $u \in \bW^{1,r}(J;D,X)$ and every $S \in (0,T]$
    the operator family $\cA(u)|_{(0,S)}$ satisfies nonautonomous
    maximal parabolic regularity on $L^r(0,S;X)$ with respect to the
    common domain $D$.
  \item The function $F$ maps $\bW^{1,r}(J;D,X)$ into $L^r(J;X)$ and
    $F-F(0)$ is Lipschitz continuous on bounded subsets of
    $\bW^{1,r}(J;D,X)$ mapping into $L^s(J;X)$ for some $s \in (r,\infty]$.
  \item Both $\cA$ and $F$ have the Volterra property.
  \item The initial value satisfies
    $u_0 \in (X,D)_{1-\frac {1}{r},r}$.
  \end{enumerate}
  Then there is a (maximal) interval
  $J_\bullet\define (0,S_\bullet) \subseteq J$ such that the
  equation~\eqref{eq:amann-abstract} admits a unique solution $u$ on
  $J_\bullet$ which belongs to the maximal regularity space
  $\bW^{1,r}(I;D,X)$ for every $I = (0,S)$ with $S \in J_\bullet$.
\end{theorem}

\begin{remark} \label{rem:amann}
  \begin{enumerate}[(i)]
  \item It is known that the Volterra property is often an essential assumption
    when dealing with nonlocal-in-time problems, both from a modeling and an
    analytical point of view, see e.g.~\cite[Ch.~V]{Gajewski74}.  We thus
    expect that the foregoing result by Amann is quite optimal in its
    essence. The reader is moreover advised to consult~\cite[Thm.~3.1]{Ama05}
    for comments on the result by its inventor and a (fixable) shortcoming in
    the proof in~\cite{Ama05b}. (This issue is already taken care of in the
    formulation of Theorem~\ref{t-Amann}.)
  \item From the very general assumptions of Theorem~\ref{t-Amann} it is clear
    that one cannot expect to obtain global-in-time solutions at all, as there
    is neither a global growth bound on the nonlinearities nor some kind of
    sign condition or monotonicity assumed. Thus we must content with a
    (maximal) local-in-time solution. See Theorem~\ref{t-quasilin-global} below
    for a global result in the $\bmW^{-1,q}_\fD(\Omega)$-setting under a
    natural sublinear growth condition.
  \item In connection with the previous point, the concepts of nonlocal-in-time
    problems and local-in-time solutions may seem formally conflicting. The
    correct way to interpret this clash of concepts is based on the
    Volterra property, and one might say that this is precisely what the
    Volterra property is there for: Let $u$ be the unique local-in-time
    solution to~\eqref{eq:amann-abstract} on $(0,S) \subseteq J$ with
    $u \in \bW^{1,r}(0,S;D,X)$. Consider the extension
    $v\define E_Su \in \bW^{1,r}(J;D,X)$ of $u$ to $J$ as in
    Lemma~\ref{lem:extension}. Then, with the assumptions of
    Theorem~\ref{t-Amann}, the problem
    \begin{equation*}
      w'(t) +[\cA(v)(t)]w(t) =F(v)(t) \quad \text{ in } X, \qquad w(0)=u_0
    \end{equation*}
    admits a unique solution $w \in \bW^{1,r}(J;D,X)$, see
    Lemma~\ref{lem:inhomo-solution}. Since $v$ and $u$ coincide on $(0,S)$, the
    Volterra property of $\cA$ and $F$ implies that $(\cA(v),F(v))$ and
    $(\cA(u),F(u))$ also coincide on $(0,S)$. From uniqueness we infer that $w$
    and $u$ are equal on $(0,S)$, and this is the way we interpret a
    local-in-time solution to a nonlocal-in-time problem with the Volterra
    property.\label{rem:amann-local}
  \item In the same way as in~\ref{rem:amann-local}, additionally shifting in
    time, one observes that the solution $u$ on $(0,S) \subseteq J$ supplied by
    Theorem~\ref{t-Amann} is also the unique solution to
    \begin{equation*}
      w'(t) +[\cA(w)(t)]w(t) =F(w)(t) \quad \text{on } (\tau,S) \text{ in } X, \qquad w(\tau)=u(\tau)
    \end{equation*}
    for every $\tau \in (0,S)$. (This is also explicitly done in the proof of
    Theorem~\ref{t-quasilin-global}.) See also Lemma~\ref{lem:maxreg-shifts}.
  \end{enumerate}
\end{remark}



We next employ the foregoing theorem to obtain sharp existence- and
uniqueness results for~\eqref{eq:system} posed in a negative
Sobolev space $\bmW^{-1,q}_\fD(\Omega)$.
To this end, let us first establish an auxiliary result; we formulate it rather
generally for possible further uses. It is used to remove a cut-off procedure
in the proof of the main theorem which allows us to rely on a G{\aa}rding
inequality assumption for the initial value only. We use the trace-extension
defined in Lemma~\ref{lem:inhomo-solution}, cf.\ Remark~\ref{rem:trace-extend}.

\begin{lemma} \label{lem:hilfs2} Let $r \in(1,\infty)$
  and suppose that $D$ and $X$ are Banach spaces with dense
  embedding $D \embeds X$ such that there exists an operator
  $\cA$ satisfying maximal parabolic regularity on $X$ with
  domain $D$. Let further $Z$ be another Banach space and
  suppose that $\Psi\colon \bW^{1,r}(J;D,X) \to L^\infty(J;Z)$ is a
  continuous Volterra map.  Let further $w \in \bW^{1,r}(J;D,X)$ and
  let $\overline{w_0} \in \bW^{1,r}(J;D,X)$ be the extension of
  $w(0)$. Suppose that $\Psi(\overline{w_0})$ is continuous at
  zero. Then, for every $\eps >0$, there
  exists a $\delta >0$ such that
    \begin{equation*}
      \Bigl\|\bigl[\Psi(w)\bigr](t) -\bigl[\Psi(\overline{w_0})\bigr](0)\Bigr\|_{Z}
      \le \eps 
      \qquad (\text{a.e.}~t \in [0,\delta]).
  \end{equation*}
\end{lemma}

\begin{proof}
  Let $\eps > 0$ and
  $\delta \in(0,T)$. Set $w_\delta \define
  E_T(w|_{(0,\delta)})$, the extension of
  $w|_{(0,\delta)} \in  \bW^{1,r}(0,\delta;D,X)$ to $\bW^{1,r}(J;D,X)$ as
  defined in Lemma~\ref{lem:extension}. By
  the Volterra property of $\Psi$ it follows that
  $\Psi(w) =\Psi( w_\delta)$ on $(0,\delta)$. We thus estimate, for
  $t \in (0,\delta)$:
  \begin{multline}\label{eq:volterra-estimate}
    \Bigl\|\bigl[\Psi(w)\bigr](t) -\bigl[\Psi
    (\overline{w_0})\bigr](0)\Bigr\|_{Z} =
    \Bigl\|\bigl[\Psi(w_\delta)\bigr](t) -\bigl[\Psi
    (\overline{w_0})\bigr](0)\Bigr\|_{Z} \\ \leq
   \Bigl\|\bigl[\Psi(w_\delta)\bigr](t) -\bigl[\Psi
   (\overline{w_0})\bigr](t)\Bigr\|_{Z}
   +
      \Bigl\|\bigl[\Psi
   (\overline{w_0})\bigr](t) -\bigl[\Psi (\overline{w_0})\bigr](0)\Bigr\|_{Z}
 \end{multline}
 Since $\Psi(\overline{w_0})$ is continuous at zero, we can choose $\delta$
 small enough that the second term in the foregoing
 inequality~\eqref{eq:volterra-estimate} is less or equal to $\eps/2$ for all
 $t \in (0,\delta)$. For the first term, we show that
 $w_\delta \to \overline{w_0}$ in $\bW^{1,r}(J;D,X)$ as $\delta \downarrow 0$. Then, by continuity of
 $\Psi$ mapping into $L^\infty(J;Z)$, the first term
 in~\eqref{eq:volterra-estimate} can also be made less or equal to $\eps/2$ by
 possibly diminishing $\delta$ further, uniformly for almost all
 $t \in (0,\delta)$, and the proof is complete.

 So, in order to prove $w_\delta \to \overline{w_0}$, 
 observe that $w_\delta-\overline{w_0}$ is the unique solution to
 \begin{equation*}
   u'(t) + \cA u(t) = f(t), \quad \text{in } X, \qquad u(0) = 0, \qquad f(t)
   \define
   \chi_{(0,\delta)}(t)\bigl[w'(t) + \cA w(t) \bigr].
 \end{equation*}
 Hence, by maximal regularity of $\cA$,
 \begin{equation*}
   \|w_\delta-\overline{w_0}\|_{\bW^{1,r}(J;D,X)}  \leq
   \bigl\|(\partial + \cA)^{-1}\bigr\|_{L^r(J;X) \to \bW_0^{1,r}(J;D,X)}\|f\|_{L^r(0,\delta;X)} \quad
   \xrightarrow{~\delta\,\downarrow\,0~} \quad 0
 \end{equation*}
 and this implies the claim.
\end{proof}

\begin{remark}
  Note that the assumption in Lemma~\ref{lem:hilfs2} regarding the
  operator $\cA$ is always fulfilled for
  \begin{equation*}
    X = \bmW^{-1,q}_\fD(\Omega), \qquad D = \bmW^{1,q}_\fD(\Omega), \qquad \cA = -\bm{\Delta}_q+1,
  \end{equation*}
  if $q \in \fI$. This follows from Theorem~\ref{thm:laplace-max-reg};
  see also Remark~\ref{r-invariantr}. Thus, for every
  $r \in (1,\infty)$ and every
  $\bmu_0 \in
  \bigl(\bmW^{-1,q}_\fD(\Omega),\bmW^{1,q}_\fD(\Omega)\bigr)_{1-\frac
    {1}{r},r}$, we have an extension
  \begin{equation*}
    \overline{\bmu_0} \in
    \bW^{1,r}(J;\bmW^{1,q}_\fD(\Omega),\bmW^{-1,q}_\fD(\Omega))
  \end{equation*}
  of $\bmu_0$ as supplied by Lemma~\ref{lem:inhomo-solution} and
  Remark~\ref{rem:trace-extend}.\label{rem:semigroup-laplacian}
\end{remark}

We will make use the foregoing Remark~\ref{rem:semigroup-laplacian} in the
second main result of this work which we next formulate in its first
version. There, we will also use $\bm{L}^q(\Omega) \define L^q(\Omega)^m$.

\textbf{Convention.}\,\,Since the domain $\Omega$ is kept fixed, we
often omit reference to $\Omega$ in the following results to remove
visual clutter where convenient and no confusion seems likely.  For
example, we then write $\bmW^{1,q}_\fD$ instead of
$\bmW^{1,q}_\fD(\Omega)$.

\begin{theorem}[Quasilinear, first version] \label{t-quasilin} Let $\lambda \geq 0$, $\gamma>0$
  and $M \geq 0$ be parameters and let $\cI_t \times \cI_x$ be
  the connected intervals from Theorem~\ref{t-extrapolatsystem}
  associated to the set $\Xi(\lambda,\gamma/2,M+1)$. Suppose
  that for some $(r,q) \in \cI_t \times \cI_x$ the following
  assumptions hold true:
  \begin{enumerate}[(i)]
  \item
    $\Phi \colon \bW^{1,r}(J;\bmW^{1,q}_\fD,\bmW^{-1,q}_\fD) \to
    L^r(J;\bmW^{-1,q}_\fD)$ is a Volterra map and $\Phi-\Phi(0)$
    is Lipschitz continuous on bounded sets mapping into
    $L^s(J;\bmW^{-1,q}_\fD)$ for some
    $s \in (r,\infty]$.
  \item For every
    $\bmu \in \bW^{1,r}(J;\bmW^{1,q}_\fD,\bmW^{-1,q}_\fD)$,
    $(\tA(\bmu)_t)_{t\in J}$ is a coefficient tensor family such
    that
    \begin{equation*}
      \tA \colon \bW^{1,r}(J;\bmW^{1,q}_\fD,\bmW^{-1,q}_\fD) \to
      L^\infty\bigl(J;L^\infty(\Omega;\cL(\C^{m(d+1)}))\bigr)
    \end{equation*}
    is a Volterra map and Lipschitz continuous on bounded sets.
  \item The initial value satisfies
    $\bmu_0 \in \bigl(\bmW^{-1,q}_\fD,\bmW^{1,q}_\fD\bigr)_{1-\frac
      {1}{r},r}$. Further, $\tA(\overline{\bmu_0})$ is
    continuous at zero with values in $L^\infty(\Omega;\cL(\C^{m(d+1)}))$
    with
    the extension  $\overline{\bmu_0}$
    of $\bmu_0$. 
  \item The system operator $\bmL_0$ associated to $\tA_0 \define
    \bigl[\tA(\overline{\bmu_0})\bigr](0)$ is
    $\lambda$-weakly elliptic in the sense of
    Definition~\ref{def:elliptic} with coercivity constant $\gamma$. Moreover,
    $\|\tA_0\| \leq M$.
  \end{enumerate}
  Then 
  the equation
  \begin{equation} \label{eq:gleidh} \bmu'(t) + \bmL(\bmu)_t\bmu(t) =\Phi(\bmu)(t) \quad \text{in } \bmW^{-1,q}_\fD, \qquad \bmu(0) = \bmu_0 
  \end{equation}
  admits a unique solution $\bmu$ on an interval
  $I=(0,S)\subseteq J$ with the regularity
  \begin{equation*}
    \bmu \in 
    \bW^{1,r}(I;\bmW^{1,q}_\fD,\bmW^{-1,q}_\fD).
  \end{equation*}
  Moreover, there exists a constant $C$ which depends only on the
  parameters $(\lambda,\gamma,M)$ such that for
  every $0 \leq \tau < \sigma \leq S$,
  \begin{equation*}
    \|\bmu\|_{\bW^{1,r}(\tau,\sigma;\bmW^{1,q}_\fD,\bmW^{-1,q}_\fD)} \leq
    C\Bigl[\|\Phi(\bmu)\|_{L^r(\tau,\sigma;\bmW^{-1,q}_\fD)} + \|\bmu(\tau)\|_{(\bmW^{-1,q}_\fD,\bmW^{1,q}_\fD)_{1-\frac
        {1}{r},r}}\Bigr].
  \end{equation*}
  If in fact $r > 2$, then $u \in C^\alpha(\overline I;\bm{L}^q(\Omega))$
  for some $\alpha > 0$.
\end{theorem}

\begin{proof}
  First of all, recall again that the extension $\overline{\bmu_0}$ of
  $\bmu_0$ to the maximal regularity space is well defined and
  available due $q \in \cI_x \subseteq \fI$, see
  Remarks~\ref{rem:extrapolation-always-iso}
  and~\ref{rem:semigroup-laplacian}.

  We next introduce a cut-off procedure to transfer the
  ellipticity assumption on $\tA_0$ to appropriate
  modifications of $\tA(\cdot)$ in order to obtain maximal
  parabolic regularity for these modifications. The cut-off will
  be removed at the end by means of Lemma~\ref{lem:hilfs2}.

  So, for given
  $\eps$, define the (Lipschitz-continuous) cut-off function
  $\kappa_\eps\colon  \C \to \overline{B_\eps(0)}$ by
  \begin{equation*}
    \kappa_\eps(s)=\begin{cases}
      s & \text{if } |s| \leq  \eps,\\
      \eps\,\frac{s}{|s|} & \text{if } |s| \ge \eps.
    \end{cases}
  \end{equation*}
  We use $\kappa_\eps$ to define a new coefficient tensor
  function $\tA_\eps$ on
  $\bW^{1,r}(J;\bmW^{1,q}_\fD,\bmW^{-1,q}_\fD)$ by
  setting
  \begin{equation*}
    \tA_\eps(\bm{w})_t \define
    \tA_0 +
    \kappa_\eps\bigl (\tA(\bm{w})_t
    -\tA_0\bigr ) \qquad (\bm{w} \in \bW^{1,r}(J;\bmW^{1,q}_\fD,\bmW^{-1,q}_\fD)), 
  \end{equation*}
  where we apply $\kappa_\eps$ component-wise for $\tA(\bm{w})_t(x)
  -\tA_0(x) \in \C^{m(d+1) \times m(d+1)}$. It follows that we can find $\eps_0
  > 0$ such that
  \begin{equation*}
    \bigl\|\tA_\eps(\bm{w})-\tA_0\bigr\|_\infty \leq 1 \qquad (\bm{w}
    \in \bW^{1,r}(J;\bmW^{1,q}_\fD,\bmW^{-1,q}_\fD),~ \eps \leq
    \eps_0), 
  \end{equation*}
  where
  $\|\cdot\|_\infty$ is the coefficient tensor family norm over
  $J$ as defined in Definition~\ref{def:coeff-tensor-family}.  In
  particular,
  $\|\tA_\eps(\bm{w})\|_\infty$ is bounded \emph{uniformly} for
  $\bm{w} \in \bW^{1,r}(J;\bmW^{1,q}_\fD,\bmW^{-1,q}_\fD)$ and $\eps
  \leq \eps_0$ by $M +1$. Moreover, since $\tA_0$ is
  $\lambda$-weak elliptic with constant $\gamma$, by choosing
  $\eps_0$ smaller if necessary, we obtain that
  \begin{equation*}
    \Re \blangle \bmL_\eps(\bm{w})_t\bmu,\bmu\brangle +
    \bigl(\lambda+\frac\gamma2\bigr)\int_\Omega |\bm{u}|^2 \geq \frac\gamma2
    \int_\Omega|\nabla \bm{u}|^2 \qquad (\bm{u} \in \bmW^{1,2}_\fD,~\eps\leq\eps_0),
  \end{equation*}
  for almost every $t \in J$, \emph{uniformly} for $\bm{w} \in
  \bW^{1,r}(J;\bmW^{1,q}_\fD,\bmW^{-1,q}_\fD)$; see Remark~\ref{r-stable}.
  Thus, overall, there exists $\eps_0 >0$ such that
  \begin{equation*}
    \tA_\eps(\bm{w}) \in \Xi(\lambda+\gamma/2,\gamma/2,M+1)\qquad (\bm{w}
    \in \bW^{1,r}(J;\bmW^{1,q}_\fD,\bmW^{-1,q}_\fD),~ \eps \leq
    \eps_0). 
  \end{equation*} 
  For $\eps \leq \eps_0$, from the assumptions on
  $(r,q)$ and Theorem~\ref{t-extrapolatsystem}, we hence find that the
  elliptic system operator family $(\bmL_\eps(\bm{w})_t)_{t\in
    J}$ induced by
  $\tA_\eps(\bm{w})$ satisfies nonautonomous maximal parabolic
  regularity on
  $L^r(J;\bmW^{-1,q}_\fD)$ with the constant spatial domain
  $\bmW^{1,q}_\fD$ for \emph{every} $\bm{w} \in
  \bW^{1,r}(J;\bmW^{1,q}_\fD,\bmW^{-1,q}_\fD)$. Since this property
  builds upon the coercivity condition~\eqref{eq:ellippt} for
  $\tA_\eps(\bm{w})$ which is \emph{uniform} in $t \in
  J$, it is straight forward to verify that we in fact obtain
  nonautonomous maximal parabolic regularity for
  $\bmL_\eps(\bm{w})$ on
  $L^r(\tau,S;\bmW^{-1,q}_\fD)$ with the constant spatial domain
  $\bmW^{1,q}_\fD$ for every interval $(\tau,S) \subseteq J$ as well.

  We note that since
  $\tA_\eps$ inherits both the Volterra- and Lipschitz-properties from
  $\tA$ via $\kappa_\eps$, the associated system operator
  $\bmL_\eps$ satisfies the assumptions in Theorem~\ref{t-Amann} with
  $D= \bmW^{1,q}_\fD$ and $X=\bmW^{-1,q}_\fD$.

  The assumptions on $\Phi$ are exactly the ones posed on $F$ in
  Theorem~\ref{t-Amann}, and the initial value $\bmu_0$ was also
  assumed to admit the correct regularity. Hence,
  Theorem~\ref{t-Amann} yields a unique maximal solution
  $\bmu_\eps$ on $(0,T_\bullet) \subseteq J$ of the problem
    \begin{equation} \label{eq:gleidh-eps} \bmu'(t) + \bmL_\eps(\bmu)_t \bmu(t) =\Phi(\bmu)(t) \quad \text{in } \bmW^{-1,q}_\fD, \qquad \bmu(0) = \bmu_0 
  \end{equation}
  with
    \begin{equation*}
    \bmu_\eps \in \bW^{1,r}(0,S;\bmW^{1,q}_\fD,\bmW^{-1,q}_\fD) \qquad (S \in (0,T_\bullet)).
  \end{equation*}

  It remains to remove the cut-off to obtain a solution to the
  original problem~\eqref{eq:gleidh}. But this is immediate by
  Lemma~\ref{lem:hilfs2} and the assumption that
  $\tA(\overline{\bmu_0})$ be continuous at zero: Use the lemma to
  determine $\delta \in (0,T_\bullet)$ such that
  \begin{equation*}
    \Bigl\|\tA(\bmu_\eps)_t -
    \tA_0 \Bigr\|_{L^\infty(\Omega;\cL(\C^{m(d+1}))}
    \le \eps 
    \qquad (\text{a.e.}~t \in [0,\delta]).
  \end{equation*}
  Then $\tA_\eps(\bmu_\eps) = \tA(\bmu_\eps)$ almost everywhere on
  $(0,\delta)$ and~\eqref{eq:gleidh-eps}
  implies that $\bmu \define \bmu_\eps$ is the unique solution to the original
  problem~\eqref{eq:gleidh} on $I \define
  (0,\delta)$ with
  \begin{equation*}
    \bmu \in \bW^{1,r}(I;\bmW^{1,q}_\fD,\bmW^{-1,q}_\fD).
  \end{equation*}
  This shows existence and uniqueness for~\eqref{eq:gleidh}. If
  $r > 2$, then the maximal regularity space embeds into
  $C^\alpha(\overline I;\bm{L}^q(\Omega))$ for some $\alpha > 0$ by
  Lemmata~\ref{l-maxparregfacts} and~\ref{lem:lq-embed}. (Recall that
  $q \in \cI_x \subseteq \fI$.) The norm estimate follows from
  $\tA(\bmu) \in \Xi(\lambda+\gamma/2,\gamma/2,M+1)$ and the
  corresponding uniform statement in Theorem~\ref{t-extrapolatsystem}
  combined with Lemma~\ref{lem:maxreg-shifts}.
\end{proof}



The integrabilities $(r,q)$ in Theorem~\ref{t-quasilin} will in general be
close to $(2,2)$ which is a certain limitation in its applicability and the
regularity obtained.  However, for these integrabilities we have the uniform
bounds on the parabolic solution operators which will become handy in
Theorem~\ref{thm:quasilinear-2} below when aiming to prove that the solutions
are in fact \emph{global-in-time}.

Under more restrictive assumptions on the time regularity for the coefficient
tensor (global continuity in time) and the integrability $r$, we next derive (H\"older-) continuity for
the solution for space dimension $d=2$. The latter is, of course, a quite
fundamental restriction related to the fact that continuous solutions in the
$\bmW^{-1,q}_\fD$-framework require $q>d$---cf.\ Lemma~\ref{l-Hoelder}---, and
the theory established in Section~\ref{Snonaut4} in general only admits
$q>2$. It is well known that H\"older-continuity for the solution requires a certain degree of
time integrability in the data which will in general not be supplied by
Theorem~\ref{t-extrapolatsystem}. The proof of Theorem~\ref{thm:quasilinear-2}
below thus builds upon nonautonomous maximal parabolic regularity with
continuous dependence on time, which will allow to have nonautonomous maximal
parabolic regularity for every time integrability:

\begin{proposition}[{\cite{Ama04a,PS01}}]
  \label{prop:pruessschnaubelt}
  Let $(\cA(t))_{t\in \overline{J}}$ be an operator family of closed
  operators on $X$ with common domain $D$ such that $\overline{J} \ni t
  \mapsto \cA(t) \in \cL(D \to X)$ is continuous. Then we have the
  following:
  \begin{enumerate}[(i)]
  \item If the family $(\cA(t))$ satisfies nonautonomous
    maximal parabolic $L^r(J;X)$-regularity for some $1<r<\infty$, then every operator
    $\cA(t)$ for $t \in \overline J$ satisfies maximal parabolic regularity on $X$.
  \item If every operator
    $\cA(t)$ for $t \in \overline J$ satisfies maximal parabolic
    regularity on $X$, then the family $(\cA(t))$ satisfies nonautonomous
    maximal parabolic $L^s(J;X)$-regularity for all $1<s<\infty$.
  \end{enumerate}
\end{proposition}


We obtain the following next existence- and uniqueness theorem, using
$\bm{C}^\alpha(\Lambda) \coloneqq C^\alpha(\Lambda)^m$:

\begin{theorem}[Quasilinear, second version] \label{thm:quasilinear-2} Let
  $d=2$. Let further $\lambda \geq 0$, $\gamma>0$ and $M \geq 0$ be parameters
  and let $\cI_t \times \cI_x$ be the connected intervals from
  Theorem~\ref{t-extrapolatsystem} associated to the set
  $\Xi(\lambda+\gamma/2,\gamma/2,M+1)$. Suppose that for some
  $q \in \cI_x$ and $r>2(1-\frac2q)^{-1}$, the following
  assumptions hold true:
  \begin{enumerate}[(i)]
  \item
    $\Phi \colon \bW^{1,r}(J;\bmW^{1,q}_\fD,\bmW^{-1,q}_\fD) \to
    L^r(J;\bmW^{-1,q}_\fD)$ is a Volterra map and $\Phi-\Phi(0)$
    is Lipschitz continuous on bounded sets mapping into
    $L^s(J;\bmW^{-1,q}_\fD)$ for some
    $s \in (r,\infty]$.
      \item For every
    $\bmu \in \bW^{1,r}(J;\bmW^{1,q}_\fD,\bmW^{-1,q}_\fD)$,
    $(\tA(\bmu))_{t\in J}$ is a coefficient tensor family such
    that
    \begin{equation*}
      \tA \colon \bW^{1,r}(J;\bmW^{1,q}_\fD,\bmW^{-1,q}_\fD) \to
      C\bigl(\overline J;L^\infty(\Omega;\cL(\C^{m(d+1)}))\bigr)
    \end{equation*}
    is a Volterra map and Lipschitz continuous on bounded sets.
  \item The initial value satisfies
    $\bmu_0 \in \bigl(\bmW^{-1,q}_\fD,\bmW^{1,q}_\fD\bigr)_{1-\frac
      {1}{r},r}$. 
  \item The system operator $\bmL_0$ associated to $\tA_0 \define
    \bigl[\tA(\overline{\bmu_0})\bigr](0)$ is
    $\lambda$-weakly elliptic in the sense of
    Definition~\ref{def:elliptic} with coercivity constant $\gamma$. Moreover,
    $\|\tA_0\| \leq M$.
  \end{enumerate}
    Then 
  the equation
  \begin{equation*} \tag{\ref{eq:gleidh}} \bmu'(t) + \bmL(\bmu)_t\bmu(t) =\Phi(\bmu)(t) \quad \text{in } \bmW^{-1,q}_\fD, \qquad \bmu(0) = \bmu_0 
  \end{equation*}
  admits a unique solution $\bmu$ on an interval
  $I=(0,S)\subseteq J$ with the regularity
  \begin{equation*}
    \bmu \in 
    \bW^{1,r}(I;\bmW^{1,q}_\fD,\bmW^{-1,q}_\fD).
  \end{equation*}
  In particular, there is $\alpha > 0$ such that
  \begin{equation*}
   u \in \bm{C}^\alpha(\overline{I \times \Omega}).
  \end{equation*}
\end{theorem}

\begin{proof}
  We proceed as in the proof of Theorem~\ref{t-quasilin} to
  obtain $\eps_0 >0$  such that
  \begin{equation*}
    \tA_\eps(\bm{w}) \in \Xi(\lambda,\gamma/2,M+1)\qquad (\bm{w}
    \in \bW^{1,r}(J;\bmW^{1,q}_\fD,\bmW^{-1,q}_\fD),~ \eps \leq
    \eps_0). 
  \end{equation*}
  This ellipticity condition is uniform in time, which means that \emph{each
    autonomous} operator $\bmL(\bmw)_\tau$ for each $\tau \in \overline J$ and
  $\bmw \in \bW^{1,r}(J;\bmW^{1,q}_\fD,\bmW^{-1,q}_\fD)$ is elliptic in the
  sense of Definition~\ref{def:elliptic} and thus satisfies the coercivity
  condition~\eqref{eq:coerc}. Thus we apply Theorem~\ref{t-extrapolatsystem}
  for every \emph{constant} coefficient tensor family
  $(\tA_\eps(\bmw)_\tau)_{t\in J}$ to find that $\bmL_\eps(\bmw)_\tau$
  satisfies maximal parabolic regularity on $\bmW^{-1,q}_\fD$ for every $\tau
  \in \overline J$. The new
  assumption regarding the coefficient tensors $\tA(\bmw)$ being continuous in
  time for each $\bmw \in \bW^{1,r}(J;\bmW^{1,q}_\fD,\bmW^{-1,q}_\fD)$ clearly
  transfers to $\tA_\eps(\bmw)$. Hence, by
  Proposition~\ref{prop:pruessschnaubelt}, we obtain that
  $\bmL_\eps(\bmw)$ satisfies \emph{nonautonomous} maximal
  parabolic $L^\rho(J;\bmW^{-1,q}_\fD)$-regularity for every $\rho \in (1,\infty)$,
  in particular for $\rho=r$.

  The rest of the proof of existence- and uniqueness is the same
  as before, which gives a unique local-in-time solution $\bmu$
  to~\eqref{eq:gleidh} on some interval $I = (0,S) \subseteq J$ with
  \begin{equation*}
    \bmu \in \bW^{1,r}(0,S;\bmW^{1,q}_\fD,\bmW^{-1,q}_\fD) \qquad (\tau
    \in I).
  \end{equation*}
  The claimed H\"older-regularity follows via
  Corollary~\ref{c-vectt} in the appendix.
\end{proof}

\begin{remark}
  \label{rem:superposition}
  At this point we want to point out a quite major implication of the different
  assumptions in Theorems~\ref{t-quasilin} and~\ref{thm:quasilinear-2}. Since
  functions in the maximal regularity space
  $\bW^{1,r}(J;\bmW^{1,q}_\fD,\bmW^{-1,q}_\fD)$ in the context of
  Theorem~\ref{t-quasilin} will in general \emph{not} be bounded on
  $J \times \Omega$ since the admitted time integrability $r$ will be too
  little even for $d=2$, the assumption that the coefficients $\tA$ be
  Lipschitz-continuous (on bounded sets) from the maximal regularity space into
  $L^\infty(J;L^\infty(\Omega;\cL(\C^{m(d+1)})))$ is quite restrictive.
  Indeed, it essentially rules out a most straightforward nonlinear case of,
  say, a superposition operator induced by a Lipschitz-continuous function
  $\C^m \to \C$, since these are only Lipschitz-continuous on spaces of bounded
  functions with a supremum-type norm. In contrast, in the context of
  Theorem~\ref{thm:quasilinear-2}, functions in the maximal regularity space
  $\bW^{1,r}(J;\bmW^{1,q}_\fD,\bmW^{-1,q}_\fD)$ \emph{will} be bounded on
  $J\times \Omega$ by the assumptions, so a superposition operator construction
  as before would in fact be feasible here. However, let us note that involving
  a smoothing operator in the nonlinearities such as a convolution can enable
  setting of Theorem~\ref{t-quasilin}, and these occur for instance in problems
  with memory effects, see e.g.~\cite[Section~§4]{Ama05}
\end{remark}

Finally we state a modification of the first theorem, Theorem~\ref{t-quasilin},
under whose assumptions one obtains global-in-time solutions. This is based on
the classical sublinear growth-bound for the right-hand side $\Phi$ and
uniformity of the parabolic solution operators as established in
Theorem~\ref{t-extrapolatsystem}. The latter essentially allows us to argue as
in the semilinear case, however with the complication that nonlocality in time
is involved. The Volterra property saves the day here.

Since for a global solution a maximal local-in-time one need be constructed
first, we also need to assume that \emph{every} coefficient tensor $\tA(\bmu)$
induces a $\lambda$-weakly elliptic system operator. Accordingly, we can get
rid of the cut-off procedure and need not assume that $\tA_0$ is continuous in
time at zero.

\begin{theorem}[Quasilinear, global version] \label{t-quasilin-global} Let
  further $\lambda \geq 0$, $\gamma>0$ and $M \geq 0$ be
  parameters and let $\cI_t \times \cI_x$ be the connected intervals from
  Theorem~\ref{t-extrapolatsystem} associated to the set
  $\Xi(\lambda,\gamma,M)$. Suppose that for some $(r,q) \in \cI_t \times \cI_x$
  the following assumptions hold true:
  \begin{enumerate}[(i)]
  \item
    $\Phi \colon \bW^{1,r}(J;\bmW^{1,q}_\fD,\bmW^{-1,q}_\fD) \to
    L^r(J;\bmW^{-1,q}_\fD)$ is a Volterra map and $\Phi-\Phi(0)$
    is Lipschitz continuous on bounded sets mapping into
    $L^s(J;\bmW^{-1,q}_\fD)$ for some
    $s \in (r,\infty]$. Moreover, there exists a
    constant $C_\Phi$ such that
    \begin{equation*}
      \|\Phi(\bmu)-\Phi(0)\|_{L^s(J;\bmW^{-1,q}_\fD)} \leq C_\Phi
      \|\bmu\|_{\bW^{1,r}(J;\bmW^{1,q}_\fD,\bmW^{-1,q}_\fD)} 
      \quad (\bmu \in \bW^{1,r}(J;\bmW^{1,q}_\fD,\bmW^{-1,q}_\fD)).
    \end{equation*}
     \item For every
    $\bmu \in \bW^{1,r}(J;\bmW^{1,q}_\fD,\bmW^{-1,q}_\fD)$,
    $(\tA(\bmu)_t)_{t\in J}$ is a coefficient tensor family such
    that
    \begin{equation*}
      \tA \colon \bW^{1,r}(J;\bmW^{1,q}_\fD,\bmW^{-1,q}_\fD) \to
      L^\infty\bigl(J;L^\infty(\Omega;\cL(\C^{m(d+1)}))\bigr)
    \end{equation*}
    is a Volterra map and Lipschitz continuous on bounded sets. Moreover, for
    every $\bmu \in \bW^{1,r}(J;\bmW^{1,q}_\fD,\bmW^{-1,q}_\fD)$, the system
    operator $\bmL(\bmu)$ is $\lambda$-weakly
    elliptic in the sense of Definition~\ref{def:elliptic} with
    coercivity constant $\gamma$ and
    $\|\Psi(\bmu)\|_\infty \leq M$.
  \item The initial value satisfies
    $\bmu_0 \in \bigl(\bmW^{-1,q}_\fD,\bmW^{1,q}_\fD\bigr)_{1-\frac
      {1}{r},r}$. 
  \end{enumerate}
  Then 
  the equation
    \begin{equation*} \tag{\ref{eq:gleidh}} \bmu'(t) + \bmL(\bmu)_t\bmu(t) =\Phi(\bmu)(t) \quad \text{in } \bmW^{-1,q}_\fD, \qquad \bmu(0) = \bmu_0 
  \end{equation*}
  admits a unique \emph{global} solution 
  \begin{equation*}
    \bmu \in 
    \bW^{1,r}(J;\bmW^{1,q}_\fD,\bmW^{-1,q}_\fD).
  \end{equation*}
  Moreover,
  there exists a constant $C$ which depends only on the
  parameters $(\lambda,\gamma,M)$ such that
  \begin{equation*}
    \|\bmu\|_{\bW^{1,r}(J;\bmW^{1,q}_\fD,\bmW^{-1,q}_\fD)} \leq
    C\Bigl[\|\Phi(\bmu)\|_{L^r(J;\bmW^{-1,q}_\fD)} + \|\bmu_0\|_{(\bmW^{-1,q}_\fD,\bmW^{1,q}_\fD)_{1-\frac
        {1}{r},r}}\Bigr].
  \end{equation*}
  If in fact $r>2$, then $u \in C^\alpha(\overline J;\bmL^q(\Omega))$
  for some $\alpha > 0$.
  
\end{theorem}

\begin{proof}
  As mentioned before, due to the new assumption that the system operator
  induced by $\tA(\bmu)$ is in the class $\Xi(\lambda,\gamma,M)$ for
  \emph{every} $\bmu \in \bW^{1,r}(J;\bmW^{1,q}_\fD,\bmW^{-1,q}_\fD)$, we can
  dispose of the cut-off procedure used in the proof of
  Theorem~\ref{t-quasilin}. Indeed, according to
  Theorem~\ref{t-extrapolatsystem}, every the operator $\bmL(\bmu)$ satisfies
  nonautonomous maximal $L^r(\tau,S;\bmW^{-1,q}_\fD)$-regularity with constant
  domain $\bmW^{1,q}_\fD$ for every $(\tau,S) \subseteq J$ under this
  assumption. Thus, the assumptions of the general Theorem~\ref{t-Amann} are
  satisfied directly and the latter yields a maximal local-in-time solution
  $\bmu$ to~\eqref{eq:gleidh} on an interval
  $J_\bullet \define (0,T_\bullet)$ which is characterized by
  \begin{equation*}
    \bmu \in \bW^{1,r}(0,S;\bmW^{1,q}_\fD,\bmW^{-1,q}_\fD) \quad (S
    \in J_\bullet), \quad \text{but} \quad \bmu \notin \bW^{1,r}(J_\bullet;\bmW^{1,q}_\fD,\bmW^{-1,q}_\fD).
  \end{equation*}  
  We are going to use the assumption on globally sublinear
  growth on $\Phi$ to show that in fact
  $\bmu \in \bW^{1,r}(J_\bullet;\bmW^{1,q}_\fD,\bmW^{-1,q}_\fD)$ which
  means that $J_\bullet = J$~(\cite[Thm.~2.1]{Ama05b}). It will
  be sufficient to show that $\bmu \in
  \bW^{1,r}(\tau,T_\bullet;\bmW^{1,q}_\fD,\bmW^{-1,q}_\fD)$ for some $\tau \in J_\bullet$.

  To this end, let $\tau \in J_\bullet$---we will choose
  it later in a suitable way---and $S \in
  (\tau,T_\bullet)$. Set $\Phi_0 \define  \|\Phi(0)\|_{L^r(0,T;\bmW^{-1,q}_\fD)}$. We claim that
  \begin{multline}\label{eq:global-self-bound}
    \|\bmu\|_{\bW^{1,r}(\tau,S;\bmW^{1,q}_\fD,\bmW^{-1,q}_\fD)}
    \lesssim
    |T_\bullet-\tau|^{\frac1r-\frac1s}\|\bmu\|_{\bW^{1,r}(\tau,S;\bmW^{1,q}_\fD,\bmW^{-1,q}_\fD)}
    \\ +
    \|\bmu(\tau)\|_{(\bmW^{-1,q}_\fD,\bmW^{1,q}_\fD)_{1-1/r,r}} + \Phi_0
  \end{multline}
  and the implicit constant is \emph{uniform} in $S$. To show this, we
  first make use of the maximal regularity space extension operator
  established in Lemma~\ref{lem:extension} and then employ the
  sublinear growth condition. (Again, recall
  $q \in \cI_x \subseteq \fI$ and
  Remark~\ref{rem:semigroup-laplacian}.)

  Indeed, note
  that for every $S \in (\tau,T_\bullet)$,
  the Volterra property implies that $\sigma_\tau\bmu \define \bmu(\cdot + \tau)$ is the unique
  solution to
  \begin{equation*}
    \bmv'(t) + \bmL(E_{S-\tau}\sigma_\tau\bmu)_t \bmv(t) =\Phi(E_{S-\tau}\sigma_\tau \bmu)(t)
    \quad \text{in } \bmW^{-1,q}_\fD, \qquad \bmv(0) = \bmu(\tau)  
  \end{equation*}
  on $(0,S-\tau)$, where $E_{S-\tau}\sigma_\tau \bmu$ is the extension of $\sigma_\tau \bmu$
  to $\bW^{1,r}(0,T;\bmW^{1,q}_\fD,\bmW^{-1,q}_\fD)$, cf.\
  Lemma~\ref{lem:extension}. Clearly,
  \begin{equation*}
    \|\sigma_\tau\bmu\|_{\bW^{1,r}(0,S-\tau;\bmW^{1,q}_\fD,\bmW^{-1,q}_\fD)}
    =  \|\bmu\|_{\bW^{1,r}(\tau,S;\bmW^{1,q}_\fD,\bmW^{-1,q}_\fD)} 
  \end{equation*}
  for every $S \in (\tau,T_\bullet)$. Using
  Theorem~\ref{t-extrapolatsystem} and
  Lemma~\ref{lem:maxreg-shifts}, we thus obtain
  \begin{multline*}
    \|\bmu\|_{\bW^{1,r}(\tau,S;\bmW^{1,q}_\fD,\bmW^{-1,q}_\fD)}
   \\ \leq
    C(\lambda,\gamma,M)\Bigl[\|\Phi(E_{S-\tau}\sigma_\tau \bmu)\|_{L^r(0,S-\tau;\bmW^{-1,q}_\fD)}
    + \|\bmu(\tau)\|_{(\bmW^{-1,q}_\fD,\bmW^{1,q}_\fD)_{1-\frac{1}{r},r}}\Bigr].
  \end{multline*}
  With the sublinear growth assumption on $\Phi$ and the uniform extension
  estimate as in  Lemma~\ref{lem:extension},
  \begin{align*}
    &\|\Phi(\sigma_\tau\bmu)\|_{L^r(0,S-\tau;\bmW^{-1,q}_\fD)} 
    \\ & \qquad \leq |S-\tau|^{\frac1r-\frac1s}
    \|\Phi(E_{S-\tau}\sigma_\tau \bmu)-\Phi(0)\|_{L^s(0,T;\bmW^{-1,q}_\fD)}
    +\Phi_0 \\ & \qquad \leq |S-\tau|^{\frac1r-\frac1s}
    C_E C_\Phi
    \Bigl[\|\bmu\|_{\bW^{1,r}(\tau,S;\bmW^{1,q}_\fD,\bmW^{-1,q}_\fD)} +
    \|\bmu(\tau)\|_{(\bmW^{-1,q}_\fD,\bmW^{1,q}_\fD)_{1-\frac{1}{r},r}}\Bigr]
    +\Phi_0.
  \end{align*}
  Hence we obtain~\eqref{eq:global-self-bound} with
  the implicit constant
  \begin{equation}\label{eq:global-const-def}
    C(\tau)) \define \bigl(1+C_E C_\Phi|T_\bullet-\tau|^{\frac1r-\frac1s}\bigr) \, C(\lambda,\gamma,M).
  \end{equation}
  and this is indeed uniform in $S$.

  Now choose $\tau$ such that $|T_\bullet -
  \tau|^{\frac1r-\frac1s}C(\tau) \leq \frac12$.  Then~\eqref{eq:global-self-bound} shows that
  \begin{equation*}
    \|\bmu\|_{\bW^{1,r}(\tau,S;\bmW^{1,q}_\fD,\bmW^{-1,q}_\fD)} \leq
    2C(\tau)\Bigl[
    \Phi_0+
    \|\bmu(\tau)\|_{(\bmW^{-1,q}_\fD,\bmW^{1,q}_\fD)_{1-\frac{1}{r},r}}\Bigr]
  \end{equation*}
  \emph{uniformly} in
  $S \in (\tau,T_\bullet)$. Hence, by
  the monotone convergence theorem,
  \begin{equation*}
    \|\bmu\|_{\bW^{1,r}(\tau,T_\bullet;\bmW^{1,q}_\fD,\bmW^{-1,q}_\fD)} =
    \limsup_{S \uparrow
      T_\bullet}\|\bmu\|_{\bW^{1,r}(\tau,S;\bmW^{1,q}_\fD,\bmW^{-1,q}_\fD)}
    < \infty
  \end{equation*}
  and it follows that
  $\bmu \in \bW^{1,r}(0,T_\bullet;\bmW^{1,q}_\fD,\bmW^{-1,q}_\fD)$. Thus $\bmu$
  is a global solution. The embedding into
  $\bm{C}^\alpha(\overline J;\bmL^q(\Omega))$ follows as before in
  Theorem~\ref{t-quasilin}.
\end{proof}

\begin{remark}
  \label{rem:hoelder-uniform}
  It would of course be of high interest to also obtain global solutions in the
  H\"older-regularity setting of Theorem~\ref{thm:quasilinear-2}. However, for
  this setting, we lack the uniform estimates for the parabolic solution
  operator which are available for $(r,q)$ close to $(2,2)$ as in
  Theorem~\ref{t-extrapolatsystem}. We propose that such a global result would
  be reachable for example if one admitted the nonlinearities to be defined on
  a H\"older space instead of the maximal regularity space and one had uniform
  estimates for the H\"older norm of the solution to nonautonomous parabolic
  system problems in some sense; in the scalar case, such a result is
  established in~\cite{MR16,HMN23}, but for systems, these kinds of 
  results are very delicate and apparently not available.
\end{remark}

\subsection{Example: Multi-species Chemotaxis}\label{sec:example}

As a possible application of the foregoing theorems, we consider a
multi-species chemotaxis model given by two coupled systems of (simplified)
Keller-Segel type going back to~\cite{KS70}. In these models one considers
species of bacteria whose movement is influenced by both natural diffusion and
interaction with a chemo-attractant; the latter may be both attraction and
repulsion. The Keller-Segel model is very well established and probably one of
the most researched biological models. We refer e.g.\ to the
surveys~\cite{Hor03,Hor04} and~\cite{ArumugamTyagi2020} for a comprehensive
overview. In a multi-species model, one would have an interaction between the
species and the respective chemo-attractants, for example, one species creating
or consuming the chemo-attractant for the other. Such multi-species problems
with particular choices of the interaction functions are widely researched, see
for instance~\cite{WinklerTao2015} and the references therein. A version of
such a multi-species system could be
\begin{equation}\label{eq:keller-couple}
  \begin{aligned}
    \partial_t u_1 - \divv(\kappa_1(u_1,v_1)\nabla u_1) 
    & = \divv\bigl(\sigma_1(u_1,v_1)(\nabla v_1-\nabla u_2)\bigr)\qquad &
    \text{in }J \times \Omega, \\
    \partial_t v_1 - \alpha_1 \Delta v_1 & = g_1(u_1,v_1,u_2,v_2) &
    \text{in } J \times \Omega, \\
    \partial_t  u_2 - \divv(\kappa_2(u_2,v_2)\nabla u_2) 
    & = \divv\bigl(\sigma_2(u_2,v_2)(\nabla v_2-\nabla u_1)\bigr)  \qquad &
    \text{in }J \times \Omega, \\
    \partial_t  v_2 - \alpha_2 \Delta v_2 & = g_2(u_1,v_1,u_2,v_2) &
    \text{in } J \times \Omega,
  \end{aligned}
\end{equation}
where the unknowns $u_i(t,x)$ and $v_i(t,x)$ describe the species density and
chemo-attractant concentration for either species $i=1,2$ at time $t$ and
spatial coordinate $x$. (These should thus be nonnegative functions with values
between $0$ and $1$.) Here we assume that the interaction between species is
manifested by either species perturbing the movement towards the
chemo-attractant for the other; this are the right-hand sides for the species
densities. Also, the dynamics for the chemo-attractants may depend on all
involved quantities. Typically, the nonlinear functions $g_i$ will be of
polynomial type. The system is to be complemented by initial values and
appropriate boundary conditions. For simplicity we suppose that the diffusion
coefficient functions $\kappa_i$ and $\sigma_i$ are real scalar functions and
that $\alpha_i > 0$. A particular feature to be highlighted here is that while
the $\kappa_i$ supposed to be (strictly) positive for every pair $(u_i,v_i)$ of
nonnegative functions, the coefficients $\sigma_i$ associated to the
chemo-attractant drift are neither restricted in sign nor in size in
general. Often considered choices for $\sigma_i$ include
$\sigma_i(u_i,v_i) = \chi \cdot u_i$ or $\chi \cdot \frac{u_i}{v_i}$ with a
constant $\chi$.

The system
tensor associated to~\eqref{eq:keller-couple} is given by
\begin{equation*}
  A(u_1,v_1,u_2,v_2) \define
  \begin{pmatrix}
    \kappa_1(u_1,v_1) &  \sigma_1(u_1,v_1) &  -\sigma_1(u_1,v_1) & 0 \\ 0 &
     \alpha_1 & 0 & 0 \\
    -\sigma_2(u_2,v_2) & 0 & \kappa(u_2,v_2) &  \sigma_2(u_2,v_2)  \\ 0 & 0 & 0 & \alpha_2
  \end{pmatrix},
\end{equation*}
where, strictly speaking, each entry in the tensor is the corresponding
multiple of the $(d \times d)$ identity matrix. Consider a subtensor of $\fA$
obtained by symmetrically removing up to three rows and columns. (That is, one
removes both column $i$ and row $i$.) Doing so, if the Legendre-Hadamard
condition is to hold for the full system at initial time, then the bilinear
form induced by the subtensor, interpreted as a real matrix as it stands, must
necessarily be positive definite. (In the Legendre-Hadamard condition as in
Remark~\ref{r-suff}, this corresponds to choosing $\eta$ with, say,
$|\eta| = 1$ fixed and setting $\zeta_i = 0$ in order to remove column and row
$i$.) It is easy to give constellations of coefficient function values for
which this will fail, for example with the subtensor obtained by removing
columns and rows three and four of $\fA$, considered at initial time:
\begin{equation*}
  \left.\begin{pmatrix}
    \kappa_1(u_1,v_1) & \sigma_1(u_1,v_1) \\ 0 & \alpha_1
  \end{pmatrix}\right|_{t=0} \enifed \begin{pmatrix}
    \kappa_1 & \sigma_1 \\ 0 & \alpha_1
  \end{pmatrix} \quad \text{with} \quad |\sigma_1| \geq 2 \sqrt{\kappa_1\alpha_1}.
\end{equation*}
Thus, under such a condition, the system tensor $\fA$
associated to~\eqref{eq:keller-couple} will fail to satisfy the
Legendre-Hadamard condition at initial time. On the other hand, at the same
time, the subtensor obtained by removing columns two and four of
$\fA$ might even satisfy the strong Legendre-condition:
\begin{equation*}
 \left.\begin{pmatrix}
    \kappa_1(u_1,v_1) & -\sigma_1(u_1,v_1) \\ - \sigma_2(u_2,v_2) & \kappa_2(u_2,v_2)
  \end{pmatrix}\right|_{t=0} \enifed \begin{pmatrix}
    \kappa_1 & -\sigma_1 \\ - \sigma_2 & \kappa_2
  \end{pmatrix}  \text{ with }  \min(\kappa_1,\kappa_2)
  - \frac{|\sigma_1+\sigma_2|}2 > 0.
\end{equation*}
This condition does not contradict the one derived for failure
of the Legendre-Hadamard condition of the other subtensor. Hence, if both are satisfied \emph{and} we manage to
implicitly solve for the chemo-attractant concentrations
$(v_1,v_2) = \cS(u_1,u_2)$ in~\eqref{eq:keller-couple} in a
suitably well-behaved dependence of the species densities
$(u_1,u_2)$, then the resulting coefficient tensor
\begin{equation*}
  \overline\fA(u_1,u_2) \define
  \begin{pmatrix}
    \kappa_1(u_1,\cS_1(u_1,u_2)) & -\sigma_1(u_1,\cS_1(u_1,u_2))
    \\ -\sigma_2(u_2,\cS_2(u_1,u_2)) & \kappa_2(u_1,\cS_1(u_1,u_2))
  \end{pmatrix}
\end{equation*}
induces a \emph{strongly elliptic} operator
$\bmL(u_1,u_2)_0 = -\nabla \cdot \overline\fA(u_1,u_2)_0\nabla$ at initial
time. This would make the full system~\eqref{eq:keller-couple}
generally amendable by the existence- and uniqueness results
from the foregoing section---under appropriate assumptions on
the nonlinearities---by first considering the reduced, strongly
elliptic system in $(u_1,u_2)$, and then re-discovering
$(v_1,v_2)$. Note for example that in space dimension $d=2$, we
have $\bmL^{q/2}(\Omega) \embeds \bmW^{-1,q}_\fD(\Omega)$ for $q>2$, so nonlinear
functions of quadratic type for $g_1,g_2$, as they occur often in
reaction terms, would fit in the
framework of Theorems~\ref{t-quasilin}
and~\ref{t-quasilin-global} with $r,q>2$, since the associated
maximal regularity spaces embed into
$C^\alpha(\overline J;\bmL^q(\Omega))$. We leave the details for future
work.

\begin{remark}
  We note that in system~\eqref{eq:keller-couple} it is
  \emph{not} feasible to implicitly solve for either species density
  $u_i$ in dependence of the other, say, $u_2$ in dependence of
  $u_1$, if one wants to rely on the existence- and uniqueness
  results from the foregoing section basing on
  Theorem~\ref{t-Amann}. This is because by solving implicitly,
  the resulting term involving $\nabla u_2$ in the equation for
  $u_1$ would become absorbed in a nonlinear function
  $\cF(u_1)$. However, the nonlinear functions $\cF$ in the context
  of Theorem~\ref{t-Amann} and the following associated ones
  must be Lipschitz-continuous with \emph{higher time integrability} than
  the sought-for solution. Since
  we work in the framework of maximal parabolic regularity here,
  an increase in time integrability of $\nabla u_2$ must be
  reflected in increased time integrability in the data for the
  $u_2$ equation which in turn involves $\nabla u_1$, creating a
  feedback loop. It thus seems that the particular
  form~\eqref{eq:keller-couple} in fact necessitates to treat
  the subsystem for the species densities $(u_1,u_2)$ as an
  actual \emph{system}.
\end{remark}

\appendix

\section{Interpolation embeddings}

In Lemma~\ref{l-maxparregfacts}, we have seen that interpolation spaces between
the domain of an operator satisfying maximal parabolic regularity and the
ambient spaces are of relevance. In the case of elliptic operators admitting
optimal regularity on $\bmW^{-1,q}_\fD$, this concerns interpolation spaces of
type $(\bmW^{-1,q}_\fD,\bmW^{1,q}_\fD)_{1-\frac {1}{r},r}$ for which we derive
some embeddings below. The results are used in the theorems in
Section~\ref{Snonaut6}. In their proofs, we make use of several
general results from interpolation theory for which we generally refer to
the comprehensive book of Triebel~\cite{Triebel78}. Recall also the set of
isomorphism indices $\fI$ for the Laplacian as in Definition~\ref{d-isoindex}.

\begin{lemma}
  \label{lem:lq-embed}
   Let $q \in \fI \cup (2,\infty)$ and $\theta \in [\frac12,1)$. Then
  \begin{equation*}
    \bigl(\bmW^{-1,q}_\fD,\bmW^{1,q}_\fD\bigr)_{\theta,1}  \embeds
   \bigl[\bmW^{-1,q}_\fD,\bmW^{1,q}_\fD\bigr]_{\frac12} \embeds
    \bm{L}^q(\Omega).
  \end{equation*}
  For $q \in \fI$, the latter embedding is in fact an equality up to equivalent norms.
\end{lemma}

\begin{proof}
  The first embedding follows from general interpolation principles due to the
  choice $\theta \geq \frac12$~(\cite[Thm.~1.10.3]{Triebel78}), so it is enough
  to show the second. 

  Denote by $\cA_q$ the part of $-\bm{\Delta}_2+1$ in
  $\bmW^{-1,q}_\fD(\Omega)$, and denote by $\bm{D}_q$ its domain equipped with
  the graph norm. If $q \in \fI$, then $\cA_q$ coincides with
  $-\bm{\Delta}_q + 1$ as in Definition~\ref{d-2}. The operator $\cA_q$ admits
  bounded imaginary powers and, denoting the domain of $\cA_q^{1/2}$ by
  $\bm{D}_q^{1/2}$, we have the square root property
  $\bm{D}_q^{1/2} = \bmL^q(\Omega)$ up to equivalent norms. This follows from
  the scalar case and the fact that the system Laplacian is of diagonal form;
  for $q \geq 2$ we refer to~\cite[Thm.~11.5]{ABHR14}, for $q < 2$ but
  $q \in \fI$ we use~\cite[][Prop.~4.6 \& Lem.~11.4]{ABHR14}
  and~\cite[Thm.~6.5]{DtER17}. Further, clearly
  $\bmW^{1,q}_\fD(\Omega) \embeds \bm{D}_q$---with equality up to equivalent norms
  if $q \in \fI$---and so using interpolation for fractional powers of operators
  with bounded imaginary powers~(\cite[Thm.~1.15.3]{Triebel78}),
  \begin{equation*}
    \bigl[\bmW^{-1,q}_\fD(\Omega),\bmW^{1,q}_\fD(\Omega)\bigr]_{\frac {1}{2}} \embeds
    \bigl[\bmW^{-1,q}_\fD(\Omega),\bm{D}_q\bigr]_{\frac {1}{2}} = \bm{D}_q^{1/2} =
    \bm{L}^q(\Omega). \qedhere
  \end{equation*}
\end{proof}

\begin{corollary} \label{c-vect-Lq} Let $q \in \fI \cup (2,\infty)$ and
  $r >2$. Then there is
  $\alpha > 0$ such that
  \begin{equation} \label{eq:interembeedeLq}
    \bigl(\bmW^{-1,q}_\fD(\Omega),\bmW^{1,q}_\fD(\Omega)\bigr)_{1-\frac {1}{r},r} \embeds
    \bmL^q(\Omega)
  \end{equation}
  and
  \begin{equation*}
    \bW^{1,r}(J;\bmW^{1,q}_\fD(\Omega),\bmW^{-1,q}_\fD(\Omega)) \embeds
    C^\alpha\bigl(\overline J;\bmL^q(\Omega)\bigr).
  \end{equation*}
\end{corollary}

\begin{proof}
  Choose $\theta \in
  (\frac12,1-\frac1r)$. Then, via
  the general interpolation principle~\cite[Thm.~1.3.3~(e)]{Triebel78} and
  Lemma~\ref{lem:lq-embed},
  \begin{equation*}
    \bigl(\bmW^{-1,q}_\fD(\Omega),\bmW^{1,q}_\fD(\Omega)\bigr)_{1-\frac {1}{r},r} \embeds
    \bigl(\bmW^{-1,q}_\fD(\Omega),\bmW^{1,q}_\fD(\Omega)\bigr)_{\theta,1} \embeds \bmL^q(\Omega).
  \end{equation*}
  This is~\eqref{eq:interembeedeLq}, and the maximal regularity embedding
  follows from combining~\eqref{eq:maxreghoelderembed} in Lemma~\ref{l-maxparregfacts} with
  Lemma~\ref{lem:lq-embed}.
\end{proof}

\begin{lemma} \label{l-Hoelder} 
  Let $q > d$ and $\alpha \in(0,1- \frac
    {d}{q})$. Then
    \begin{equation} \label{eq:heold}
      \bigl(\bmW^{-1,q}_\fD(\Omega),\bmW^{1,q}_\fD(\Omega)\bigr)_{\theta,1} \embeds
      \bm{C}^\alpha(\overline\Omega),
    \end{equation}
    for $\theta \in (\frac12+\frac {d}{2q},1)$ and
    $\alpha = 2\theta -1-\frac{d}q$.
  \end{lemma}

\begin{proof}
   Let
  $\theta$ and $\alpha$ be as assumed in the statement.  Then
  $\tau \define 2\theta-1 = \alpha + \frac{d}q$, and with the
  reiteration theorem~(\cite[Thm.~1.10.3.2]{Triebel78}) we have
  \begin{equation} \label{eq:intembedd}
    \bigl(\bmW^{-1,q}_\fD(\Omega),\bmW^{1,q}_\fD(\Omega)\bigr)_{\theta,1} = \Bigl
    (\bigl[\bmW^{-1,q}_\fD(\Omega),\bmW^{1,q}_\fD(\Omega)\bigr]_{\frac {1}{2}},
    \bmW^{1,q}_\fD(\Omega)\Bigr )_{\tau,1}.
  \end{equation}
  By virtue of Lemma~\ref{lem:lq-embed}, the complex
  interpolation space embeds into $\bmL^q(\Omega)$. We thus continue
  in~\eqref{eq:intembedd} to obtain
  \begin{equation} \label{eq:heold1}
    \bigl(\bmW^{-1,q}_\fD(\Omega),\bmW^{1,q}_\fD(\Omega)\bigr)_{\theta,1} \embeds
    \bigl(\bmL^q(\Omega),\bmW^{1,q}_\fD(\Omega)\bigr)_{\tau,1}, \qquad \tau = \alpha + \frac{d}{q}.
  \end{equation}
  It remains to show that the right hand side embeds into the asserted H\"older
  space. Heuristically, the interpolation space should be a fractional Sobolev
  type space whose differentiability order $\tau$ is larger than $d/q$ and
  which should thus embed exactly into the H\"older space
  $\bm{C}^\alpha(\overline\Omega)$ with $\alpha = \tau-d/q$. However, our
  assumptions on $\Omega$ and $\fD$ are not strong enough to prove this
  directly, so we take a detour. (See~\cite{BechtelEgert2019} for a
  comprehensive and precise interpolation theory, allowing to argue in this way
  under slightly stronger assumptions, though.)

  More precisely, we make use of the classical
  relations~(\cite[Ch.~2.4.2~(16) \& Thm.~2.8.1~(4)]{Triebel78})
  \begin{equation}\label{eq:besov-embed}
    \bigl(\bmW^{1,q}(\R^d),\bmL^q(\R^d)\bigr)_{\tau,1} =  \bm{B}^\tau_{q,1}(\R^d) \embeds
    \bm{C}^\alpha(\R^d),
  \end{equation}
  where $\bm{B}^\tau_{q,1}(\R^d)$ is a Besov space and we have tacitly used the
  boldface notation to denote $\C^m$-valued function spaces. We move from $\Omega$ to
  Euclidean space 
  by employing the continuous linear
  extension operator $\bm{E} \colon \bmW^{1,q}_\fD \to \bmW^{1,q}(\R^d)$ which
  simultaneously also provides a continuous extension operator
  $\bm{L}^q(\Omega) \to L^q(\R^d)$. It was already mentioned in
  Remark~\ref{r-fortsetz} that such an operator exists in the given geometric
  setting, and it to estimate, for all
  $\bmu \in \bmW^{1,q}_\fD(\Omega)$, as follows:
  \begin{align*}  \|\bmu \|_{\bm{C}^\alpha(\overline\Omega)} \leq
    \|\bm{E} \bmu \|_{\bm{C}^\alpha(\R^d)}& \lesssim \|\bm{E} \bmu \|
    _{\bm{B}^\tau_{q,1}(\R^d)} \\ &\lesssim \|\bm{E} \bmu\|_{\bmL^q(\R^d)} ^{1-\tau} \|\bm{E}
    \bmu\|_{\bmW^{1,q}(\R^d)} ^{\tau} \lesssim \|\bmu\|_{\bmL^q(\Omega)} ^{1-\tau}
    \|\bmu\|_{\bmW^{1,q}_\fD(\Omega)} ^{\tau}.
  \end{align*}
  The multiplicative inequality used in the middle estimate follows from the
  first embedding in~\eqref{eq:besov-embed}, this is a fundamental property of
  interpolation functors, cf.~\cite[Thm.~1.3.3~(g)]{Triebel78}. Finally, the
  multiplicative inequality obtained overall indeed implies the embedding
  $(\bmL^q(\Omega),\bmW^{1,q}_\fD(\Omega))_{\tau,1} \embeds
  \bm{C}^\alpha(\overline\Omega)$~(\cite[Lem.~1.10.1]{Triebel78}), and this was the claim.
\end{proof}

Note that the foregoing interpolation result is optimal in the sense
that one cannot relax the conditions on $\theta$ and $\alpha$ even if
the geometrical setting is assumed to be arbitrarily smooth; in
particular, it cannot be relaxed even for $\fD = \partial\Omega$ or,
associated, considering the corresponding function spaces on
$\R^d$. 


\begin{corollary} \label{c-vectt} Assume $q>d$ and
  $r >2 ( 1-\frac {d}{q})^{-1}$. Then there is
  $\alpha > 0$ such that
  \begin{equation} \label{eq:interembeede}
    \bigl(\bmW^{-1,q}_\fD(\Omega),\bmW^{1,q}_\fD(\Omega)\bigr)_{1-\frac {1}{r},r} \embeds
    \bm{C}^\alpha(\overline \Omega)
  \end{equation}
  and
  \begin{equation*}
    \bW^{1,r}(J;\bmW^{1,q}_\fD(\Omega),\bmW^{-1,q}_\fD(\Omega)) \embeds
    \bm{C}^\alpha(\overline {J \times \Omega}).
  \end{equation*}
\end{corollary}

\begin{proof}
  Choose 
  $\theta \in(\frac {1}{2}+\frac {d}{2q},1-\frac
  {1}{r})$ and proceed as in the proof of
  Corollary~\ref{c-vect-Lq} to obtain, for $\alpha =
  2\theta-1-\frac{d}{q}>0$, from Lemma~\ref{l-Hoelder},
  \begin{equation*}
    \bigl(\bmW^{-1,q}_\fD(\Omega),\bmW^{1,q}_\fD(\Omega)\bigr)_{1-\frac {1}{r},r} \embeds
    \bigl(\bmW^{-1,q}_\fD(\Omega),\bmW^{1,q}_\fD(\Omega)\bigr)_{\theta,1} \embeds \bm{C}^\alpha(\overline\Omega).
  \end{equation*}
  The remaining part follows as in the proof of Corollary~\ref{c-vect-Lq}.
\end{proof}

\section*{Acknowledgements}

The author appreciates Joachim Rehberg~(WIAS Berlin) and Karoline
Disser~(Universit\"at Kassel) for discussions about the subject, and
is grateful to Alexander Keimer (FAU Erlangen-N\"urnberg) for
suggestions on notation.

\section*{Declaration of interest}

None.

\printbibliography

\end{document}